\documentclass[11pt]{article}
\usepackage{color}
\usepackage[dvips, outer=2cm, inner=2cm, top=2cm, bottom=2cm]{geometry}                
\usepackage[parfill]{parskip}    
\usepackage{graphicx}
\usepackage{amsfonts}
\usepackage{amsmath}
\usepackage{amsthm}
\usepackage{amssymb}
\usepackage{epstopdf}
\newtheorem{thm}{Theorem}[section]

\newtheorem{prop}[thm]{Proposition}

\newtheorem{lemma}[thm]{Lemma}

\newtheorem{cor}[thm]{Corollary}

\DeclareMathOperator{\tr}{tr}
\DeclareMathOperator{\sym}{Sym}

\author{Salim Ali Altu\u{g}}
\date{}

\begin{document}

\title{Beyond Endoscopy via the Trace Formula - I\\ Poisson Summation and Contributions of Special Representations}

\maketitle

\begin{abstract}

With analytic applications in mind, in particular Beyond Endoscopy (\cite{L1}), we initiate the study of the elliptic part of the trace formula. Incorporating the approximate functional equation to the elliptic part we control the analytic behavior of the volumes of tori that appear in the elliptic part. Furthermore by carefully choosing the truncation parameter in the approximate functional equation we smooth-out the singularities of orbital integrals. Finally by an application of Poisson summation we rewrite the elliptic part so that it is ready to be used in analytic applications, and in particular in Beyond Endoscopy. As a by product we also isolate the contributions of special representations as pointed out in \cite{L1}.

\end{abstract}

\begin{section}{Introduction}

 The Arthur-Selberg trace formula is (arguably) the most general tool in the theory of automorphic forms up to current date. Its development into the current form has taken over half a century and in the mean time it has given rise to many spectacular results on the functoriality conjectures (see for example \cite{A} \S25,26 and \cite{A2}). Almost all of these results go through a comparison of trace formulae on different groups coupled with local harmonic analysis. Although these results being very successful, they only coover a limited number of special cases of the functoriality conjectures, and in general the conjectures are wide open.
 
Relatively recently (in \cite{L1}) a new strategy, which is now known as ``Beyond Endoscopy'', was introduced to attack the general functoriality conjectures. Very roughly it can be described as a two step process: First step is to isolate, by means of the trace formula, the (packets of) cuspidal automorphic representations whose $L$-functions (for a representation of the dual group) have the same order of pole at $s=1$. The second step involves a comparison of this data for two different groups and aims at determining functorial transfers. The method, in particular, proposes a new and non-comparative use of the trace formula. In this paper we will only be concerned with the first of the two steps. The central problem of the first step is to understand the asymptotic behavior of certain averages of trace formulae on a single group with varying test functions (cf. \eqref{1bullet}).

In \cite{L1} the study of these averages was initiated for the group $GL(2)$ and symmetric power representations (cf. \S\ref{beoverview}). At the heart of these averages are the terms coming from the so-called ``elliptic part" of the trace formula (cf. equation \eqref{2eq4}). The elliptic part involves averages of orbital integrals weighed by certain arithmetic data (eg. volumes of tori) varying in families. The highly irregular behavior of these quantities on top of the singularities of orbital integrals make the analysis troublesome. We also note that in \cite{L1} the elliptic part, although numerically analyzed, was not treated.

This paper lays the foundations of a method to study the elliptic part of the trace formula in analytic problems. We introduce the approximate functional equation to the elliptic part in order to resolve the problems of arithmetic and analytic nature at once. We then go on and isolate the contribution of special representations in the elliptic part (cf. \S\ref{obstacles}). Finally we end up with an expression for the elliptic part that is ready to use in analytic applications, particularly in Beyond Endoscopy. The results of this paper will then be used in the subsequent papers (\cite{Ali1} and \cite{Ali2}) where we execute the first non-trivial case of Beyond Endoscopy via the trace formula, and prove bounds\footnote{More precisely we will reprove the classical $\tfrac14$-bound of Kuznetsov via the trace formula.} towards the Ramanujan conjectures respectively.

In order to state our results more precisely and to put them into context, in the next few paragraphs we will briefly go over the idea in \cite{L1}. We will then state the main results of this paper in Theorem \ref{mainthm}.

\begin{subsection}{A Brief Overview of Beyond Endoscopy}\label{beoverview}

In order to simplify notation and to keep the analogy with \cite{L1}, we will be only working over the field $\mathbb{Q}$. Let us begin by describing the general idea of Beyond Endoscopy.

 Let $S$ be a finite set of primes including the archimedean place and $\pi$ be a cuspidal automorphic representation of $G$ unramified outside of $S$. For $p\notin S$, let $A(\pi_p)\in {^L}G$ be the local parameter of $\pi_p$. Finally let $\rho$ be a finite dimensional representation of ${^L}G$. Recall that to this data one can attach the incomplete\footnote{The missing factors for the primes in $S$ are expected to not to effect the analytic behavior of the automorphic $L$-functions.} automorphic $L$-function (cf. \cite{Bo} for details) defined by
\begin{align*}
L^S(s,\pi,\rho):&=\prod_{p\notin S}\det\left(1-\rho\left(A(\pi_p)\right)\cdot p^{-s}\right)^{-1}\\
:&=\sum_{\substack{n\\ \gcd(n,S)=1}}\tfrac{a_{\pi,\rho}(n)}{n^s}\tag{$\bullet$}\label{bullet0}
\end{align*}
Taking negative of the logarithmic derivative of $L^S(s,\pi,\rho)$ we see that the asymptotic expansion of the partial averages\footnote{Note that $\mathfrak{S}_{\pi}(X)$ depends on the chosen finite set of primes $S$. For our purposes we will chose $S$ once and for all, therefore we dropped it from the notation. If the need to emphasize the choice of $S$ arises we will write $\mathfrak{S}_{\pi,\rho}(X,S)$ instead of $\mathfrak{S}_{\pi,\rho}(X)$.}
\begin{align*}
\mathfrak{S}_{\pi,\rho}(X):&=\sum_{\substack{p<X\\ p\notin S}}\log(p)a_{\pi,\rho}(p)
\end{align*}
in terms of powers $X^{\beta}, \, \Re(\beta)\geq1$, give us the location and multiplicity of the poles of $L^S(s,\pi,\rho)$ on and to the right of $\Re(s)=1$. Moreover for certain test functions $f^{p,\rho}_v\in C^{\infty}(G(\mathbb{Q}_v))$ at $v\notin S$ (cf. \S\ref{sectest} or \cite{L1} pg.19), and for arbitrary\footnote{We can also allow functions which are not necessarily compactly supported however this is not the main issue here.} $f_v\in C_c^{\infty}(\mathbb{Q}_v)$, we can express the average of $a_{\pi}(p)$ weighted by $\prod_{v\in S}\tr(\pi_v(f_v))$ as the trace of the operator $R(f^{p,\rho})$ (see \cite{A} pg.7 for the definition of $R(f)$) on the cuspidal part of the spectrum, where $f^{p,\rho}:=\prod_{v\in S}f_v\prod_{v\notin S}f_v^{p,\rho}$. i.e.
\[\sum_{\pi}a_{\pi,\rho}(p)\prod_{v\in S}\tr(\pi_v(f_v))=\tr(R_{cusp}(f^{p,\rho}))\]
 In the above expression we have denoted the orthogonal projection of $R(f^{p,\rho})$ to the cuspidal spectrum by $R_{cusp}(f^{p,\rho})$. The idea in \cite{L1} is to study the asymptotic behavior of 
 \begin{align*}
\mathfrak{S}_{\rho}(X):&=\sum_{\pi}\sum_{\substack{p<X\\ p\notin S}}\log(p)a_{\pi,\rho}(p)\prod_{v\in S}\tr(\pi_v(f_v))\\
&=\sum_{\substack{p<X\\ p\notin S}}\log(p)\tr(R_{cusp}(f^{p,\rho}))\tag{$\bullet\bullet$}\label{1bullet}
 \end{align*}
 by using the trace formula to re-express $\tr(R_{cusp}(f^{p,\rho}))$. At this moment let us pause momentarily to make some comments. 
 
 Firstly, we would like to note that for a $\pi$ which satisfies the Ramanujan conjectures (\cite{Sa1}) $L^S(s,\pi,\rho)$ is holomorphic in the region $\Re(s)>1$, and it is expected that the only possible poles appear at the point $s=1$. Therefore for $\pi$ of the this type, the leading term of the asymptotic expansion of $\mathfrak{S}_{\pi,\rho}(X)$ will only have terms of size $X$. On the other hand one expects that the $\pi$ that violate the Ramanujan conjectures to be functorial transfers from smaller groups and thus can be understood inductively. Therefore one can focus the attention on representations satisfying Ramanujan conjectures (which are called ``Ramanujan type" in \cite{L1}) and study the coefficient of the term $X$ in the asymptotic expansion of $\mathfrak{S}_{\pi,\rho}(X)$. This was the approach taken in \cite{L1}. 

Secondly, we would like to make a brief historical remark. Right after the idea of Beyond Endoscopy came out Sarnak, in his letter to Langlands (\cite{S}), suggested studying a variant of $\mathfrak{S}_{\pi,\rho}(X)$ (for the group $GL(2)$ and $\rho=\sym^k$, the symmetric $k$'th power representation) where the sum over $p$ is replaced by a sum that runs over integers, and to use the Petersson-Kuznetsov formula (\cite{IK} \S16.4), a relative trace formula, to analyze the resulting expressions. For $k\leq 2$ these modifications conveniently allowed the asymptotic expansions of (the modified) $\mathfrak{S}_{\sym^k}(X)$ to be studied (see \cite{Ven} for a treatment of $k=1,2$ and \cite{Her} for related results). For higher $k$ serious analytic problems arise and an analysis has not been carried yet, for more details on this we refer the reader to \cite{S}.

After this detour, we now go back to $\mathfrak{S}_{\rho}(X)$ and \cite{L1}. Since in this paper we will only be considering $GL(2)$ and symmetric power representations, for what follows let us fix an integer $k>0$ and use the notation: $\rho=\rho_k=\sym^k$, $\mathfrak{S}_{\rho}(X)=\mathfrak{S}_k(X)$, $f^{p,\rho}=f^{p,k}$. As we have already noted in the first part of the introduction, a detailed study of $\mathfrak{S}_{\rho}(X)$, for $G=GL(2)$ and $\rho=\sym^k$, was initiated in \cite{L1} (pg.17-34). There the contribution to \eqref{1bullet} of all of the terms but the elliptic ones had been analyzed.

\begin{subsection}{Obstacles in the study of the elliptic part}\label{obstacles}

We have indicated in the introduction that the main difficulties that make the analysis of the elliptic terms not straightforward are the appearance of class numbers of quadratic extensions in various families (cf. equation \eqref{2eq2}) and the singularities of orbital integrals (\S\ref{secsing}). 

Additional complications are caused by contributions of certain special representations. More precisely, as we have remarked in the previous paragraphs, the most fundamental part of the asymptotic expansion of $\mathfrak{S}_k(X)$ is expected to be the term of order $X$, which corresponds to the contribution of those forms of Ramanujan type. However the trace formula\footnote{In \cite{L1} Langlands uses the trace formula in \cite{JL} rather than Arthur's trace formula and we will follow this choice.} expresses the trace of the operator $R(f^{p,\rho})$ on the \emph{discrete part of the spectrum} as a sum of geometric and spectral terms (\cite{JL} pg.271-272). Thus in order to study the asymptotic behavior of $\mathfrak{S}_{k}(X)$ one needs to isolate in the geometric side the contribution of those representations that are not of Ramanujan type. An important example (which is the only example in the setting considered in \cite{L1}) of this is the trivial representation, which we denote by \textbf{1}. Its trace, $\tr(\textbf{1}(f^{p,k}))$, has to be isolated in the geometric side (this contribution occurs in the elliptic part, cf. Theorem \ref{mainthm}) of the trace formula before one can use it to study $\mathfrak{S}_{k}(X)$. Furthermore we would like to isolate this in such a way that the resulting expression is in a form that is suitable for further analysis of $\mathfrak{S}_{k}(X)$.

We would like to note that the contribution of $\tr(\textbf{1}(f^{p,k}))$ was previously studied in \cite{L1} and \cite{FLN}. In \cite{L1} the contribution was approximated and numerical experiments were done for the resulting expression for $\mathfrak{S}_{k}(X)$. In the more recent paper \cite{FLN} the contribution of $\tr(\textbf{1}(f))$, for a general class of functions $f$, was isolated for a general group\footnote{These properties exclude $G=GL(n)$, however their argument can easily be extended to cover this case too.} $G$ that is semi-simple, simply connected and satisfs $G=G_{der}$, where $G_{der}$ denotes the derived group of $G$. In that paper the authors perform Poisson summation on what they call the "Hitchin-Steinberg basis" and identify the contribution of $\tr(\textbf{1}(f))$ as the main term on the dual sum. However their approach, so far, has not allowed a further study of the resulting expression after removing $\tr(\textbf{1}(f))$. Our method in this paper is similar to the one in \cite{FLN} that we also use Poisson summation. The main difference, which allows us to go further and get an expression that is suitable for analysis, is that we use the approximate functional equation (equation \eqref{afe}) in treating the class numbers (i.e. volumes of tori in \cite{FLN}), which amounts to an additive truncation rather than the multiplicative truncation that is used in \cite{FLN}.

Returning back to our discussion, we would finally like to note that in \cite{L1} (pg. 25) the contribution to \eqref{1bullet} of the residues of Eisenstein series\footnote{This is the contribution to the trace formula of term $(vi)$ of \cite{JL}.} (denoted in \cite{L1} by $\tr(\xi_0(f_m^p))$, which in our notation will be denoted by $\tr(\xi_0(f^{p,k}))$) was analyzed and shown to contribute $\alpha_kX+o(X)$ to $\mathfrak{S}_{k}(X)$, where $\alpha_k\in\mathbb{C}$ is given in equations (31) and (32) of \cite{L1}. It was stated there that one expects this contribution to appear in the elliptic part, however this was not shown up to current date. This contribution is also isolated in Theorem \ref{mainthm}.

\end{subsection}

\begin{subsection}{Results of this paper}

In this paper we analyze the elliptic part of the trace formula, isolate the contributions of the special representations that were mentioned in the previous paragraph, and rewrite it in a form that is suitable for analytic applications, in particular for Beyond Endoscopy. In order to state the result we will need to introduce some notation which is explained in detail in \S\ref{sec2}.

Let $G:=GL(2)$ and $\mathbb{A}_{\mathbb{Q}}$ denote the ring of ad\`{e}les of $\mathbb{Q}$. To keep the analogy with \cite{L1} and to avoid notional complications we will be considering automoprhic representations of $G$ over $\mathbb{Q}$ which are unramified at every finite place and whose central characters are trivial on $\mathbb{R}^{\times}_{>0}\hookrightarrow \mathbb{A}_{\mathbb{Q}}^{\times}$. Let $p\in\mathbb{Z}_{>0}$ be a prime and $k\in\mathbb{Z}_{>0}$ be an integer. Let us denote the scalar matrices with positive real entries by $Z_+$. Let $f^{p,k}=f_{\infty}\cdot f_p^{p,k}\prod_{q\neq p}f_{q}^{p,k}$, where $f_{\infty}\in C^{\infty}(Z_+\backslash G(\mathbb{R}))$ and $f_q^{p,k}$ are as in \S\ref{sectest}. We note that for a cuspidal automorphic representation, $\pi$ with central character as above, this choice of test functions satisfy $\tr(\pi(f_{\infty}))\cdot a_{\pi,\rho_k}(p)=\tr(\pi(f^{p,k}))$, where $a_{\pi,\rho_k}(p)$ is still defined by \eqref{bullet0}. 
 
 Let $G(\mathbb{Q})^{\#}$ denote the set of conjugacy classes in $G(\mathbb{Q})$, and $\gamma_{(4y,x)}\in G(\mathbb{Q})^{\#}$ denote the conjugacy class of elements having trace $x$ and determinant $y$. We define the functions $\theta_{\infty}^{\mp}\in C_c(\mathbb{R})$ by $\theta_{\infty}^{\mp}(x):=2|x^2\pm1|\cdot Orb(f_{\infty}; \gamma_{(\mp1,x)})$, where the orbital integral, $Orb(f_{\infty}; \gamma_{(\mp1,x)})$, is as defined in \S\ref{measnorm}. Let $F,H_{0},H_1\in C^{\infty}(\mathbb{R})$ be as in equations \eqref{afedefF}, \eqref{H+} and \eqref{h-} of \S\ref{secafe} respectively. Finally let us denote the elliptic part of the trace formula for the test function $f^{p,k}$ by $\tr(R(f^{p,k}))_{ell}$. 

Then we have the following:

\begin{thm}\label{mainthm}Let $1>\alpha>0$, and $\upsilon>0$ be any number such that $\zeta(u+1)$ does not have any zeros for $|u|<\upsilon$. Let
$\mathcal{C}_{\upsilon}=\begin{setdef}{(0,it)}{t\in (-\infty,-\upsilon)\cup(\upsilon,\infty)}\end{setdef}\cup C_{\upsilon}$, and $C_{\upsilon}$ denotes the left-half of the circle of radius $\upsilon$ around $0$. Then,

\begin{align*}
\tr(R(f^{p,k}))_{ell}&=\tr(\textbf{1}(f^{p,k}))-\tr(\xi_0(f^{p,k}))-\Sigma(\square)-\tfrac{k+1}{2}\sum_{\mp}\int_{x^2\pm1>0}\tfrac{\theta_{\infty}^{\mp}(x)}{\sqrt{|x^2\pm1|}}dx\\
&+\tfrac{p^{k/2}}{2}\sum_{\mp}\left\{\int\theta_{\infty}^{\mp}\left(x\right)\left[\tfrac{2}{\pi i}\int_{(-1)}\tilde{F}(u)\left(\tfrac{(4p^k)^{-\alpha}}{|x^2\pm1|^{\alpha}}\right)^{-u}\tfrac{\zeta(2u+2)}{\zeta(u+2)}\tfrac{\left(1-p^{-(u+1)(k+1)}\right)}{\left(1-p^{-(u+1)}\right)}du\right.\right.\\
&\left.\left.\hspace{0.7in}+\tfrac{\sqrt{\pi}p^{-k/2}}{\sqrt{|x^2\pm1|}}\tfrac{1}{\pi i}\int_{\mathcal{C}_{\upsilon}}\tilde{F}(u)\tfrac{\Gamma\left(\frac{\iota_{(x^2\pm1)}+u}{2}\right)}{\Gamma\left(\frac{\iota_{(x^2\pm1)}+1-u}{2}\right)}\left(\tfrac{\pi (4p^k)^{\alpha-1}}{|x^2\pm1|^{1-\alpha}}\right)^{-u}\tfrac{\zeta(2u)}{\zeta(u+1)}\tfrac{\left(1-p^{-u(k+1)}\right)}{\left(1-p^{-u}\right)}du\right]dx\right\}\\
&+\tfrac{p^{k/2}}{2}\sum_{\mp}\sum_{f=1}^{\infty}\tfrac{1}{f^3}\sum_{l=1}^{\infty}\tfrac{1}{l^2}\sum_{\substack{\xi\in\mathbb{Z}\\ \xi\neq0}}Kl_{l,f}(\xi,\mp p^k)\\
&\hspace{1in}\times\left\{\int\theta_{\infty}^{\mp}\left(x\right)\left[F\left(\tfrac{lf^2(4p^k)^{-\alpha}}{|x^2\pm1|^{\alpha}}\right)+\tfrac{lf^2p^{-k/2}}{2\sqrt{|x^2\pm1|}}H\left(\tfrac{lf^2(4p^k)^{\alpha-1}}{|x^2\pm1|^{1-\alpha}}\right)\right]e\left(\tfrac{-x\xi p^{k/2}}{2 l f^2}\right) dx\right\} 
\end{align*}
Where\footnote{For $q\neq p$ a prime, $Kl_{q,1}(\xi,p^k)$ with $\gcd(\xi,q)=1$ is the classical Kloosterman sum $S(\bar{2}\xi,2\xi p^k;q)$ (cf. \cite{S} equation (70)), hence the notation.},

\begin{align*}
Kl_{l,f}(\xi,\mp p^{k}):&=\sum_{\substack{a\bmod 4lf^2\\ a^2\pm4p^k\equiv 0\bmod f^2\\ \tfrac{a^2\pm4p^k}{f^2}\equiv0,1\bmod 4}}\left(\tfrac{(a^2\pm4p^k)/f^2}{l}\right)e\left(\tfrac{a\xi}{4lf^2}\right)\\
\iota_{x^2\pm1}:&=\begin{cases}0&\text{if $x^2\pm1>0$}\\ 1&\text{if $x^2\pm1<0$}\end{cases}\\
H\left(\tfrac{lf^2(4p^k)^{\alpha-1}}{|x^2\pm1|^{1-\alpha}}\right)&=H_{\iota_{(x^2\pm1)}}\left(\tfrac{lf^2(4p^k)^{\alpha-1}}{|x^2\pm1|^{1-\alpha}}\right)
\end{align*}
$H_{0}, H_1$ being defined in Corollary \ref{approxfuneq}, and
\begin{multline*}
\Sigma(\square):=\sum_{\mp}\sum_{\substack{m\in\mathbb{Z}\\ m^2\pm4p^k= \square}}\theta_{\infty}^{\mp}\left(\tfrac{m}{2p^{k/2}}\right)\sum_{\substack{f^2\mid m^2\pm4p^k}}^{'}\tfrac{1}{f}\sum_{l=1}^{\infty}\tfrac{1}{l}\left(\tfrac{(m^2\pm4p^k)/f^2}{l}\right)\\ 
\times\left[F\left(\tfrac{lf^2}{|m^2\pm4p^k|^{\alpha}}\right)+\tfrac{l f^2}{\sqrt{|m^2\pm4p^k|}}H\left(\tfrac{lf^2}{|m^2\pm4p^k|^{1-\alpha}}\right)\right]
\end{multline*}
Where the $'$ on top of the summation sign means the sum is running over $f \mid (m^2\pm4p^k)$ such that $\tfrac{m^2\pm4p^k}{f^k}\equiv0,1\bmod4$.

\end{thm}

Since it is easy to lose track in the somehow overwhelming notation above, we would like to clarify the following: 
\begin{itemize}
\item The function $f^{p,k}$ are the test functions that are used in the trace formula to arrive at \eqref{1bullet}. 
\item $f_{\infty}\in Z_+\backslash GL_2(\mathbb{R})$ is an arbitrary smooth function. The property $f_{\infty}$ that its orbital integrals are compactly supported is to ensure some generality that may be useful for applications\footnote{For instance one may wish to take $f_{\infty}$ to be a matrix coefficient of some discrete series representation.}. For all practical purposes of the paper one can take $f_{\infty}$ as compactly supported itself. 
\item $F$ is the test function that we choose for the approximate functional equation, and $H_0,H_1$ are transforms of $F$ that appear in the approximate functional equation (cf. \S\ref{secafe}). The explicit choice made in \eqref{afedefF} is for conveniently realizing the Mellin transform, $\tilde{F}$, and its analytic properties (cf. Lemma \ref{lemmaF}). The arguments go through with an arbitrary choice of a Schwarz class function $F$. 
\end{itemize}

We would also like note that although Theorem \ref{mainthm} is stated for the automorphic representations with the ramification and central character restrictions given above, the methods are robust and they easily generalize to cover the most general case.

\end{subsection}

\end{subsection}

\end{section}

\section*{Acknowledgements} We would like to thank professors Robert Langlands and Peter Sarnak for introducing the author to the subject and for very many invaluable conversations throughout the evolution of this work. We would also like to thank Prof. Matthew Young for pointing out the functional equation in equation \eqref{approx1}, and to Ng\^o Bao Ch\^au, Eddie Herman, Arul Shankar and Jacob Tsimerman for many discussions.

\begin{section}{Preliminaries and the Trace Formula}\label{sec2}

In this section we will review the setup of \cite{L1} in more detail. We will first describe the set of automorphic representations that will be of interest to us. We will then fix measure normalizations and review the appropriate choice of test functions to arrive at \eqref{1bullet}. Then we will recall their their orbital integrals as well as the volume factors that appear in the trace formula. Finally we will review the singularities of (archimedean) orbital integrals which will be central to the analysis.

Throughout the paper, unless otherwise explicitly stated, $e(x)$ will denote $e^{2\pi i x}$, $\left(\tfrac{D}{\cdot}\right)$ will denote the Kronecker symbol and $\sqrt{\cdot}$ will mean the positive branch of the square-root function.

\begin{subsection}{The relevant sets of automorphic representations}\label{relevantsets}

Let $G:=GL(2)$ and $\mathbb{A}=\mathbb{A}_{\mathbb{Q}}$ be the ring of ad\`{e}les of $\mathbb{Q}$. We will be interested in automorphic representations $\pi$ of $G(\mathbb{A})$ where $\pi_p$ is unramified for every finite prime $p$, and whose central characters are trivial on $\mathbb{R}^{\times}_{>0}\hookrightarrow \mathbb{A}^{\times}$. Let us denote those matrices in the center of $G(\mathbb{R})$ having positive entries by $Z_+$. Then we can, and will, identify $\mathbb{R}^{\times}_{>0}$ with $Z_+$.  We remark that  since we are insisting $\pi_p$ to be unramified at every finite place and the central character to be trivial on $Z_+$, by strong approximation, the central character of the representation $\pi$ is necessarily trivial (as observed in \cite{L1}).

\end{subsection}

\begin{subsection}{The Trace Formula}

\begin{subsubsection}{Elliptic part of the trace formula and measure normalizations}\label{measnorm}

An element, $\gamma\in G(\mathbb{Q})$, will be called\footnote{The notion of an elliptic element depends on the choice of the field. However since we have fixed the base field to be $\mathbb{Q}$ we drop this from notation and simply say elliptic instead of elliptic over $\mathbb{Q}$. When the time arises to distinguish a field $K$ we will use the notation ``elliptic over $K$".} \emph{elliptic} if its characteristic polynomial is irreducible over $\mathbb{Q}$. For $\gamma\in G(\mathbb{Q})$ let $G_{\gamma}$ denote the centralizer of $\gamma$ in $G$. We also let $G(\mathbb{Q})^{\#}$ to denote the set of $\mathbb{Q}$-conjugacy classes in $G$, and $G(\mathbb{Q})^{\#, ell}$ denote the set of elliptic conjugacy classes. The elliptic part of the trace formula is the sum
\[\sum_{\substack{\gamma\in G(\mathbb{Q})^{\#,ell}}} vol(\gamma)\cdot\prod_{q}Orb(f_q;\gamma)\]
Where,
\begin{align*}
Orb(f_q;\gamma):&=\int_{G_{\gamma}(\mathbb{Q}_q)\backslash G(\mathbb{Q}_q)} f_q(g^{-1}\gamma g)d\bar{g}_q\\
vol(\gamma):&=\int_{Z_+G_{\gamma}(\mathbb{Q})\backslash G_{\gamma}(\mathbb{A})}dg
\end{align*}
and the product over $q$ runs through all the primes including $\infty$.

Measures in the above integrals are normalized as follows\footnote{We note here that the only reason to choose this normalization is to keep the analogy with \cite{L1}. There are more natural choices of measures on both $G$ and the tori (for example see \cite{FLN}).}: On $G$, at a non-archmedean prime $p$ we choose the Haar measure on $G(\mathbb{Q}_p)$ giving measure $1$ to $G(\mathbb{Z}_p)$, and at $\infty$ we choose any Haar measure (the explicit choice is not important for our purposes here). On $G_{\gamma}$ we normalize the measures in a similar manner: 
\begin{itemize}
\item At a non-archimedean prime $p$ we choose the Haar measure giving measure $1$ to $G_{\gamma}(\mathbb{Z}_p)$.
\item At $\infty$, any $\delta\in G(\mathbb{R})$ can be decomposed as $\delta=z_{\delta}\bar{\delta}u_{\delta}$, where $z_{\delta}\in Z_+$ is the central matrix with entries $\sqrt{|\det(\delta)|}$, $u_{\delta}=\left(\begin{smallmatrix}sign(\det(\gamma))&\\&1\end{smallmatrix}\right)$, and $\bar{\delta}\in SL_2(\mathbb{R})$. 

\begin{itemize}

\item If $\gamma\in G(\mathbb{Q})$ that is elliptic over  $\mathbb{R}$ (i.e. has two \emph{non-real} eigenvalues), and let the eigenvalues of $\bar{\delta}\in G_{\gamma}(\mathbb{R})$ be $e^{i\theta},e^{-i\theta}$. We take the measure to be $d\theta$.

\item If $\gamma\in G(\mathbb{Q})$ that is split over $\mathbb{R}$ (i.e. has two distinct \emph{real} eigenvalues), and let the eigenvalues of $\bar{\delta}\in G_{\gamma}(\mathbb{R})$ be $\lambda, \lambda^{-1}$. We take the measure to be $\frac{d \lambda}{\lambda}$.

\end{itemize}
\end{itemize}

\end{subsubsection}

\begin{subsubsection}{Test functions, orbital integrals and volumes of tori}\label{sectest}

In this subsection we quickly go over the relevant choices of test functions to reach \eqref{1bullet}. The details of the calculations of orbital integrals and volumes of tori can be found in pg.19-21 of \cite{L1}. Let $p$ be a prime and $k>0$ be an integer. Let $\rho=\sym^k$ be the symmetry $k$'th power representation of ${^L}G=GL(2,\mathbb{C})$.

For a finite prime $q$ and an integer $r\geq0$ let us first define $f_q^{(r)}\in C_c(\mathbb{Q}_q)$ to be the characteristic function of the set
\[\begin{setdef}{X\in Mat_{2\times 2}(\mathbb{Z}_q)}{|\det(X)|_q=q^{-r}}\end{setdef}\]
Where $Mat_{2\times2}(\mathbb{Z}_q)$ denote the set of two-by-two matrices with coefficients in $\mathbb{Z}_q$ and $|\cdot|_q$ denotes the $q$-adic absolute value on $\mathbb{Q}$. Now let $f_q^{p,k}\in C_c^{\infty}(\mathbb{Q}_q)$ be defined by
\begin{itemize}
\item If $q$ is  finite prime such that $q\neq p$, then $f_q^{p,k}:=f_q^{(0)}$.
\item At $p$: $f_p^{p,k}:=p^{-k/2}f_p^{(k)}$.
\item At $\infty$: $f_{\infty}\in C^{\infty}(Z_+\backslash G(\mathbb{R}))$ is such that its orbital integrals are compactly supported, and other than this condition it is arbitrary. 
\end{itemize}
Finally define $f^{p,k}$ by
\[f^{p,k}:=f_{\infty}\cdot f_p^{p,k}\cdot\prod_{\substack{q\neq p}}f_q^{p,k}\]
Let $\gamma\in G(\mathbb{Q})$ be elliptic and let us denote $(4\det(\gamma),\tr(\gamma))$ by $(N_{\gamma},m_{\gamma})$. Let $m_{\gamma}^2-N_{\gamma}=s_{\gamma}^2D_{\gamma}$, where $D_{\gamma}$ is the discriminant of the quadratic number field $\mathbb{Q}\left(\sqrt{m_{\gamma}^2-N_{\gamma}}\right)$. Normalizing the measures as above, the computations\footnote{In \cite{L1} it is assumed that $\mathbb{Q}(\sqrt{m_{\gamma}^2-N_{\gamma}})\neq \mathbb{Q}(\sqrt{-2})$ or $\mathbb{Q}(\sqrt{-3})$ however the calculations easily generalize to cover those cases.} on pg.17-18 of \cite{L1} gives
\begin{equation}\label{2eq1}
vol(\gamma)=\begin{cases}2h(\gamma)R(\gamma)&\text{if $D_{\gamma}>0$}\\ \frac{2\pi h(\gamma)}{\omega_{\gamma}}&\text{if $D_{\gamma}<0$}\end{cases}
\end{equation}
Where $\omega_{\gamma}, h(\gamma), R(\gamma)$ are the number of roots of unity, the class number and the regulator of $\mathbb{Q}\left(\sqrt{m_{\gamma}^2-N_{\gamma}}\right)$ respectively. Then following Lemma 1 of \cite{L1} we see that if $\det(\gamma)=\pm p^k$ then
\begin{equation}\label{2eq2}
vol(\gamma)\cdot \prod_{q}Orb(f_q^{p,k};\gamma)=p^{-k/2}vol(\gamma)\cdot Orb(f_{\infty};\gamma) \cdot \left\{\sum_{f\mid s_{\gamma}}f\prod_{q\mid f}\left(1-\left(\tfrac{D_{\gamma}}{q}\right)q^{-1}\right)\right\}
\end{equation}
and the left hand side is $0$ otherwise (cf. equation (60) of \cite{L1}). Recall that by Dirichlet's class number formula we have,
\[L\left(1,\left(\tfrac{D_{\gamma}}{\cdot}\right)\right)=\begin{cases}\tfrac{2h(\gamma)R(\gamma)}{\sqrt{D_{\gamma}}}&\text{if $D_{\gamma}>0$}\\ \frac{2\pi h(\gamma)}{\omega_{\gamma}\sqrt{|D_{\gamma}|}}&\text{if $D_{\gamma}<0$}\end{cases}\]
Combining this with \eqref{2eq1} we get
\[vol(\gamma)=\sqrt{|D_{\gamma}|}L\left(1,\left(\tfrac{D_{\gamma}}{\cdot}\right)\right)\]
Substituting this into \eqref{2eq2} gives
\begin{equation*}
vol(\gamma)\cdot \prod_{q}Orb(f_q^{p,k};\gamma)=p^{-k/2} Orb(f_{\infty};\gamma) \sqrt{|D_{\gamma}|}L\left(1,\left(\tfrac{D_{\gamma}}{\cdot}\right)\right) \left\{\sum_{f\mid s_{\gamma}}f\prod_{q\mid f}\left(1-\left(\tfrac{D_{\gamma}}{q}\right)q^{-1}\right)\right\}
\end{equation*}
Finally by using the change of variables $f\mapsto \tfrac{s_{\gamma}}{f}$ and rearranging the terms we get
\begin{equation}\label{2eq3}
vol(\gamma)\cdot \prod_{q}Orb(f_q^{p,k};\gamma)= Orb(f_{\infty};\gamma) \tfrac{|m_{\gamma}^2-N_{\gamma}|^{1/2}}{p^{k/2}}\sum_{f\mid s_{\gamma}}\tfrac{1}{f}L\left(1,\left(\tfrac{(m_{\gamma}^2-N_{\gamma})/f^2}{\cdot}\right)\right) 
\end{equation}
when $\det{\gamma}=\pm p^k$ and the left hand side vanishes otherwise.

\end{subsubsection}

\begin{subsubsection}{Archimedan orbital integrals and their singularities}\label{secsing}
We will recall the asymptotic behavior of archimedean orbital integrals\footnote{The non-archimedean orbital integrals have exactly the same type of singularities however since we will only be considering representations that are unramified at every finite place the archimedean case will be sufficient for our purposes.} in our context. For a more detailed introduction see \cite{K}, \cite{Lab} and \cite{Sh} and references therein.

Let $f_{\infty}\in C_c^{\infty}(Z_+\backslash G(\mathbb{R}))$ be as above and $\gamma\in G(\mathbb{Q})$ be a regular semisimple element (i.e. centralizer has minimal dimension). We are interested in the behavior of $Orb(f;\gamma)$ as $\gamma$ approaches a central element $z\in G(\mathbb{Q})$. As is described in \cite{L1} (page 21 equation (26)) and\footnote{A quick look shows that our measure normalizations on the tori are the same as the ones given in \cite{Sh} up to a constant. } \cite{Sh} (equation (1)), around $z\in Z(\mathbb{R})$: There exists a Weyl group invariant neighborhood, $N_z$, of $z$ and smooth functions $g_1,g_2\in C_c^{\infty}(N_z)$ (depending on the function $f_{\infty}$ and the point $z$) such that 
\[Orb(f_{\infty};\gamma)=g_1(\gamma)+\tfrac{|\gamma_1\gamma_2|^{\frac{1}{2}}}{|\gamma_1-\gamma_2|}g_2(\gamma)\tag{$\star$}\label{2bullet}\]
where $\gamma_1,\gamma_2$ are the eigenvalues of $\gamma$. Furthermore $g_1$ is supported only on the elliptic torus. We also remark that as $\gamma$ approaches a central element the orbital integral, $Orb(f_{\infty};\gamma)$, has a singularity of the prescribed form, $\tfrac{|\gamma_1\gamma_2|^{1/2}}{|\gamma_1-\gamma_2|}$, and that $\tfrac{\gamma_1\gamma_2}{(\gamma_1-\gamma_2)^2}$ is the discriminant function of $G$. 

We will now re-express \eqref{2bullet} in terms of the $(N_{\gamma},m_{\gamma})$ coordinates as in the previous section. Recall that $m_{\gamma}=\tr(\gamma)$, and $N_{\gamma}=4\det(\gamma)$. The discriminant then becomes
\[\tfrac{(\gamma_1-\gamma_2)^2}{\gamma_1\gamma_2}=4\left(\tfrac{m_{\gamma}^2}{N_{\gamma}}-1\right)\]
Then in the $(N,m)$ coordinates the asymptotic expansion of the orbital integral can be re-expressed as
\[Orb(f_{\infty};\gamma)=g_1(m_{\gamma},N_{\gamma})+\tfrac{1}{2}\left|\tfrac{m_{\gamma}^2}{N_{\gamma}}-1\right|^{-1/2}g_2(N_{\gamma},m_{\gamma})\]
Where $g_1,g_2$, by abuse of notation, denotes the corresponding functions in the $(N,m)$ coordinates. 

Also recall that $f_{\infty}$ is assumed to be invariant under $Z_+$, therefore we have $f_{\infty}(z\gamma)=f_{\infty}(g)$ for any $z\in Z_+$. This in particular implies that $g_i(a^2N,am)=g_i(N,m)$ for any $a\in\mathbb{R}^+$ and $i=1,2$. By taking $a=\sqrt{|N|}$ $g_1$ and $g_2$ depend only on the ratio $\tfrac{m}{\sqrt{|N|}}$ and the sign of $N$. Therefore the orbital integrals can be expressed as
\[Orb(f_{\infty};\gamma)=g_1^{sign(N_{\gamma})}\left(\tfrac{m_{\gamma}}{\sqrt{|N_{\gamma}|}}\right)+\tfrac{1}{2}\left|\tfrac{m_{\gamma}^2}{N_{\gamma}}-1\right|^{-1/2}g_2^{sign(N_{\gamma})}\left(\tfrac{m_{\gamma}}{\sqrt{|N_{\gamma}|}}\right)\tag{$\star\star$}\label{2star}\]
where $g_i^{\mp1}(x):=g_i(\mp1,x)$. We also remark that by the note following \eqref{2bullet}, when $sign(N_{\gamma})<0$, the torus $G_{\gamma}$ is split at $\infty$, and $g_1$ vanishes.

\end{subsubsection}

\begin{subsubsection}{Final form of the elliptic part}

Recall that the elliptic part of the trace formula is the sum
\[\sum_{\substack{\gamma\in G(\mathbb{Q})^{\#,ell}}} vol(\gamma)\cdot\prod_{q}Orb(f_q;\gamma)\]
By \eqref{2eq3} this is 
\[\sum_{\substack{\gamma\in G(\mathbb{Q})^{\#,ell}\\ \det(\gamma)=\pm p^{k}}}Orb(f_{\infty};\gamma) \tfrac{|m_{\gamma}^2-N_{\gamma}|^{1/2}}{p^{k/2}}\sum_{f\mid s_{\gamma}}\tfrac{1}{f}L\left(1,\left(\tfrac{(m_{\gamma}^2-N_{\gamma})/f^2}{\cdot}\right)\right) \]
Also recall that only the $\gamma$ for which $N_{\gamma}=4\det(\gamma)=\pm 4p^k$ give a non-zero contribution to the sum above. Therefore $p^{k/2}=\tfrac{\sqrt{|N_{\gamma}|}}{2}$. Hence by \eqref{2star},
\begin{align*}
Orb(f_{\infty};\gamma) \tfrac{|m_{\gamma}^2-N_{\gamma}|^{1/2}}{p^{k/2}}&=2Orb(f_{\infty};\gamma) \left|\tfrac{m_{\gamma}^2}{N_{\gamma}}-1\right|^{1/2}\\
&=2\left|\tfrac{m_{\gamma}^2}{N_{\gamma}}-1\right|^{1/2}g_1^{sign(N_{\gamma})}\left(\tfrac{m_{\gamma}}{\sqrt{|N_{\gamma}|}}\right)+g_2^{sign(N_{\gamma})}\left(\tfrac{m_{\gamma}}{\sqrt{|N_{\gamma}|}}\right)\\
&=\theta_{\infty}^{sign(N_{\gamma})}\left(\tfrac{m_{\gamma}}{\sqrt{|N_{\gamma}|}}\right)
\end{align*}
Where,
\begin{equation*}
\theta_{\infty}^{\mp}(x):=2|x^2\pm1|^{1/2}g_1^{\mp}(x)+g_2^{\mp}(x)\tag{$\star\star\star$}\label{2starstar}
\end{equation*}

Finally, recall that the conjugacy classes in $GL(2)$ are parametrized by their determinant and trace, and a conjugacy class corresponding to determinant $n$ and trace $m$ is elliptic if and only if $m^2-4n\neq \square\in \mathbb{Q}$. Since with our choice of test functions the only contribution to the elliptic part is from $\gamma$ with $\det(\gamma)=\pm p^{k}$, the elliptic part can be written as
\begin{equation}\label{2eq4}
\sum_{\mp}\sum_{\substack{m\in \mathbb{Z}\\ m^2\pm4p^k\neq\square}}\theta_{\infty}^{\mp}\left(\tfrac{m}{2p^{k/2}}\right)\sum_{\substack{f^2\mid m^2\pm 4p^k}}^{'}\tfrac{1}{f}L\left(1,\left(\tfrac{(m^2\pm 4p^k)/f^2}{\cdot}\right)\right)
\end{equation}
Where the $'$ on top of the summation sign indicates that the sum over $f$ is over the square divisors of $m^2\pm 4p^k$ such that $\tfrac{m^2\pm4p^k}{f^2}$ is a discriminant, i.e. $\tfrac{m^2\pm4p^k}{f^2}\equiv 0,1\bmod 4$.

\end{subsubsection}

\end{subsection}

\end{section}

\begin{section}{Approximate Functional Equation}

This section is dedicated to the derivation of an approximate functional equation for the weighted sum of the $L$-values that appear in \eqref{2eq4}. We will first review the functional equation that the sum over $f$ of the $L$-values satisfy. The point to pay attention is that the weights (i.e. the $f$-sum) in \eqref{2eq4} are arranged so that the $f$-sum as a whole satisfies a convenient functional equation. Once we have the functional equation we will derive an approximate functional equation in a routine manner. For most of the material on the approximate functional equation we will follow \S10.4 of \cite{IK}.

\begin{subsection}{A Functional Equation}

Let $\delta\in\mathbb{Z}\backslash \{0\}$ be a discriminant, i.e. $\delta\equiv0,1\bmod 4$, and let $\left(\tfrac{\delta}{\cdot}\right)$ denote the Kronecker symbol. As usual let $L\left(z,\left(\tfrac{\delta}{\cdot}\right)\right)$ denote the Dirichlet $L$-function associated to the character $\left(\tfrac{\delta}{\cdot}\right)$. i.e.
\[L\left(z,\left(\tfrac{\delta}{\cdot}\right)\right)=\sum_{l=1}^{\infty}\tfrac{1}{l^z}\left(\tfrac{\delta}{l}\right)\]
Let $L(z,\delta)$ be defined by
\begin{equation}\label{L}
L(z,\delta):=\sum_{\substack{f^2\mid \delta}}^{'}\tfrac{1}{f^{2z-1}}L\left(z,\left(\tfrac{\delta/f^2}{\cdot}\right)\right)
\end{equation}
Where the $'$ on top of the summation sign, once again, means that the sum is running over $f$ such that $\delta/f^2\equiv 0,1\bmod 4$. Let $\Lambda(z,\delta)$ be the completed $L$-function, i.e.
\[\Lambda(z,\delta):=\left(\tfrac{|\delta|}{\pi}\right)^{\frac{z}{2}}\Gamma\left(\tfrac{z+\iota_{\delta}}{2}\right)L(z,\delta)\]
Where $e_{\delta}$ is defined by
\[\iota_{\delta}=\begin{cases}0&\delta>0\\ 1&\delta<0\end{cases}\tag{\#}\label{sharp}\]
Then the completed $L$-function satisfies the following functional equation:

\begin{prop}\label{approxprop0}
\begin{equation}\label{approx1}
\Lambda(z,\delta)=\Lambda(1-z,\delta)
\end{equation}
In particular we have
\[L(z,\delta)=\left(\tfrac{|\delta|}{\pi}\right)^{\frac12-z}\tfrac{\Gamma\left(\frac{1-z+\iota_{\delta}}{2}\right)}{\Gamma\left(\frac{z+\iota_{\delta}}{2}\right)}L(1-z,\delta)\tag{$6'$}\label{approx1'}\]
\end{prop}

\begin{proof}This is the content of Lemma 2.1 of \cite{SY}. We only note that in the indicated reference it is implicitly assumed that $\delta/f^2$ is a discriminant. It turns out that this functional equation was also observed earlier by several other authors in related contexts (see for instance Bykovskii, \cite{By}, and Zagier, \cite{Za}). We refer the reader to the proof of Lemma 2.1 of \cite{SY} and the references in \S2 of the same reference for more on the history.

\end{proof}

\end{subsection}

\begin{subsection}{Approximate Functional Equation}\label{secafe}

In what follows we will derive an approximate functional equation for $L(z,\delta)$. Everything in this section is standard and we include this section to keep the treatment self contained. We will take almost all of this material from Chapter 10, \S10.4 of \cite{IK}. 

Let $F\in C^{\infty}(\mathbb{R}^{+})$ be
\begin{equation}\label{afedefF}
F(x)=\tfrac{1}{2K_0(2)}\int_{x}^{\infty}e^{-y-\tfrac{1}{y}}\tfrac{dy}{y}
\end{equation}
Where $K_s(z)$ denotes the $s$'th modified Bessel function of the second kind. Then,\\

\begin{lemma}For every $x>0$ we have
\begin{equation}\label{afe5}
0<F(x)<\tfrac{e^{-x}}{2K_0(2)}
\end{equation}
and
\begin{equation}\label{afe6}
0<1-F(x)<\tfrac{e^{-\tfrac{1}{x}}}{2K_0(2)}
\end{equation}
\end{lemma}

\begin{proof}
\cite{IK} pg. 257.
\end{proof}

Let $\tilde{F}(z)$ denote the Mellin transform of $F$. i.e. 
\begin{equation}\label{mel}
\tilde{F}(z)=\int_0^{\infty}F(u)u^{z}\tfrac{du}{u}
\end{equation}
We have the following lemma about the analytic behavior of $\tilde{F}$:\\

\begin{lemma}\label{lemmaF}Explicitly; $\tilde{F}(z)=\tfrac{1}{z}\tfrac{K_z(2)}{K_0(2)}$. It is holomorphic except for a simple pole at $z=0$ with residue $1$. Furthermore, $\tilde{F}(z)$ is odd, and for $z=\sigma+it\in\mathbb{C}$ we have the uniform bound $\tilde{F}(z)\ll |z|^{|\sigma|-1}e^{-\tfrac{\pi}{2}|t|}$

\end{lemma}
\begin{proof}
\cite{IK} pg. 257-258.

\end{proof}

\begin{prop}[Approximate functional equation]\label{approximatefunceq} Let $\delta\in\mathbb{Z}$ be a discriminant (i.e. $\delta\equiv0,1\bmod4$) and $L(z,\delta)$ be defined by \eqref{L}. Then for any $z\in\mathbb{C}$ we have,

\[L(z,\delta)=\sum_{\substack{f^2\mid \delta}}^{'}\tfrac{1}{f^{2z-1}}\sum_{l=1}^{\infty}\tfrac{1}{l^z}\left(\tfrac{\delta/f^2}{l}\right)F\left(\tfrac{lf^2}{A}\right)
+\left(\tfrac{|\delta|}{\pi}\right)^{\frac12-z}\sum_{\substack{f^2\mid \delta}}^{'}\tfrac{1}{f^{1-2z}}\sum_{l=1}^{\infty}\tfrac{1}{l^{1-z}}\left(\tfrac{\delta/f^2}{l}\right)H_{\delta,z}\left(\tfrac{lf^2A}{|\delta|}\right)\]
Where,
\[H_{\iota_{\delta},z}(y):=\tfrac{\pi^{z-\frac12}}{2\pi i}\int_{\Re(u)=1}\tfrac{\Gamma\left(\frac{1+u-z+\iota_{\delta}}{2}\right)}{\Gamma\left(\frac{z-u+\iota_{\delta}}{2}\right)}(\pi y)^{-u}\tilde{F}\left(u\right)du\]
\end{prop}

\begin{proof}

Let $\tilde{F}$ be as in \eqref{mel}. For an arbitrary parameter $A>0$ consider
\[\tfrac{1}{2\pi i}\int_{\Re(u)=\sigma}L(z+u,\delta)\tilde{F}(u)A^udu\]
Where $\sigma$ is such that $\sigma +\Re(z)>1$, and therefore the integral and the sum defining the $L$-function are absolutely convergent. Interchanging the integral and the sum  and using the Mellin inversion formula gives
\begin{align*}
\sum_{\substack{f^2\mid \delta}}^{'}\tfrac{1}{f^{2z-1}}\sum_{l=1}^{\infty}\tfrac{1}{l^z}\left(\tfrac{\delta/f^2}{l}\right)F\left(\tfrac{lf^2}{A}\right)&=\tfrac{1}{2\pi i}\int_{\Re(u)=\sigma}L\left(z+u,\delta\right)A^u\tilde{F}\left(u\right)du\\
&=\tfrac{1}{2\pi i}\int_{\Re\tilde{F}(u)=\sigma}L(z+u,\delta)A^u\tilde{F}\left(u\right)du
\end{align*}
Then shifting the contour to $\Re(u)=\sigma'<0$ picks up the pole of $\tilde{F}(u)$ at $u=0$ and gives
\begin{equation*}
\tfrac{1}{2\pi i}\int_{\Re(u)=\sigma}L(z+u,\delta)A^u\tilde{F}\left(u\right)du=L(z,\delta)+\tfrac{1}{2\pi i}\int_{\Re(u)=\sigma'}L(z+u,\delta)A^u\tilde{F}\left(u\right)du
\end{equation*}
Using the change of variables $u\mapsto -u$ and using the oddness of $\tilde{F}$ transforms the $\sigma'$-integral to
\[\tfrac{1}{2\pi i}\int_{\Re(u)=\sigma'}L(z+u,\delta)A^u\tilde{F}\left(u\right)du=-\tfrac{1}{2\pi i}\int_{\Re(u)=\sigma'}L(z-u,\delta)A^{-u}\tilde{F}\left(u\right)du\]
Finally using the functional equation \eqref{approx1'} gives
\begin{equation*}
\tfrac{1}{2\pi i}\int_{\Re(u)=\sigma'}L(z-u,\delta)A^{-u}\tilde{F}\left(u\right)du=
\tfrac{1}{2\pi i}\int_{\Re(u)=-\sigma'}\left(\tfrac{|\delta|}{\pi}\right)^{\frac12+u-z}\tfrac{\Gamma\left(\frac{1+u-z+\iota_{\delta}}{2}\right)}{\Gamma\left(\frac{z-u+\iota_{\delta}}{2}\right)}L(1-z+u,\delta)A^{-u}\tilde{F}\left(u\right)du
\end{equation*}
Therefore we get,

\[L\left(z,\delta\right)=\sum_{\substack{f^2\mid \delta}}^{'}\tfrac{1}{f^{2z-1}}\sum_{l=1}^{\infty}\tfrac{1}{m^z}\left(\tfrac{\delta/f^2}{l}\right)F\left(\tfrac{lf^2}{A}\right)
+|\delta|^{\frac12-z}\sum_{\substack{f^2\mid \delta}}^{'}\tfrac{1}{f^{1-2z}}\sum_{l=1}^{\infty}\tfrac{1}{l^{1-z}}\left(\tfrac{\delta/f^2}{l}\right)H_{\iota_\delta,z}\left(\tfrac{lf^2A}{|\delta|}\right)\]
We note that in the statement of the proposition we took $\sigma'=1$ for convenience.

\end{proof}

\begin{cor}\label{approxfuneq}Let $\delta\in\mathbb{Z}$ be a discriminant (i.e. $\delta\equiv0,1\bmod4$) and $L(1,\delta)$ be defined by \eqref{L}. Then,
\begin{equation}\label{afe}
L\left(1,\delta\right)=
\sum_{\substack{f^2\mid \delta}}^{'}\tfrac{1}{f}\sum_{l=1}^{\infty}\tfrac{1}{l}\left(\tfrac{\delta/f^2}{l}\right)\left[F\left(\tfrac{lf^2}{A}\right)+\tfrac{ lf^2}{\sqrt{|\delta|}}H_{\iota_{\delta}}\left(\tfrac{lf^2A}{|\delta|}\right)\right]\end{equation}
Where $\iota_{\delta}$ is as defined in \eqref{sharp}, and
\begin{align*}
H_{0}(y):&=H_{0,1}(y)\\
&=\tfrac{\sqrt{\pi}}{2\pi i}\int_{\Re(u)=1}\tfrac{\Gamma\left(\frac{u}{2}\right)}{\Gamma\left(\frac{1-u}{2}\right)}(\pi y)^{-u}\tilde{F}\left(u\right)du\tag{$H_{0}$}\label{H+}\\
H_{1}(y):&=H_{1,1}(y)\\
&=\tfrac{\sqrt{\pi}}{2\pi i}\int_{\Re(u)=1}\tfrac{\Gamma\left(\frac{1+u}{2}\right)}{\Gamma\left(\frac{2-u}{2}\right)}(\pi y)^{-u}\tilde{F}\left(u\right)du\tag{$H_1$}\label{h-}
\end{align*}
\end{cor}

\end{subsection}

\begin{subsection}{Estimates on $H_{\iota_{\delta}}$}

We have the following bound on $H_{\iota_{\delta}}$:\\

\begin{lemma}\label{afelem2}For any $\Re(x)\geq 1$ we have
\begin{equation}\label{afe9}
H_{\iota_{\delta}}(x)\ll \tfrac{1}{x}e^{-2\sqrt{x}}
\end{equation}
Where the implied constant is absolute.
\end{lemma}
\begin{proof}The only difference between $H_{0}$ and $H_{1}$ is the difference in the $\Gamma$-factors, so we start with bounding those. Recall Stirling's approximation (cf. pg.326 of \cite{SS}):
\[\Gamma(u)=\sqrt{2\pi}\tfrac{u^u}{\sqrt{u}e^u}\left(1+O\left(\tfrac{1}{\sqrt{|u|}}\right)\right)\]
Using this we get
\begin{align*}
\tfrac{\Gamma\left(\frac{u}{2}\right)}{\Gamma\left(\frac{1-u}{2}\right)}&=\left(\tfrac{u}{2e}\right)^{u-\frac{1}{2}}\left(\tfrac1u-1\right)^{\frac u2}\left(1+O\left(\tfrac{1}{\sqrt{|u|}}\right)\right)\\
\tfrac{\Gamma\left(\tfrac{1+u}{2}\right)}{\Gamma\left(\tfrac{2-u}{2}\right)}&=\left(\tfrac{u}{2e}\right)^{u-\tfrac{1}{2}}\left(1+\tfrac{1}{u}\right)^{\tfrac{u}{2}}\left(\tfrac{2}{u}-1\right)^{\tfrac{u-1}{2}}\left(1+O\left(\tfrac{1}{\sqrt{|u|}}\right)\right)
\end{align*}
Note that the map $u\mapsto \tfrac{1}{u}-1$ maps the line $\Re(u)=1$ to the circle centered at $-1/2$ on the real line, with radius $1/2$, and therefore we have $|\tfrac1u-1|\leq1$. Similarly we get $|1+\tfrac1u|\leq2$ and $|\tfrac2u-1|\leq1$. These inequalities then imply that for $\Re(u)=1$,

\[\tfrac{\Gamma\left(\frac{u}{2}\right)}{\Gamma\left(\frac{1-u}{2}\right)},
\tfrac{\Gamma\left(\tfrac{1+u}{2}\right)}{\Gamma\left(\tfrac{2-u}{2}\right)}\ll \left(\tfrac{u}{\sqrt2e}\right)^{u-\frac12}\tag{$\dagger$}\label{bul}\]
Where the implied constant is absolute. Substituting \eqref{bul} into the definitions for $H_{\iota_{\delta}}(x)$, and using the bound on $\tilde{F}(x)$ given in Lemma \ref{lemmaF} we get
\begin{align*}
H_{\iota_{\delta}}(x)&\ll \int_{(1)}|\tfrac{u}{\sqrt2\pi e}|^{\Re(u)-\tfrac{1}{2}}|x|^{-\Re(u)}|u|^{\Re(u)-1}e^{-\tfrac{\pi |\Im(u)|}{2}}du\\
&\ll\int_{(1)}|u|^{2\Re(u)-\tfrac{3}{2}}| e^2x|^{-\Re(u)}e^{-\tfrac{\pi |\Im(u)|}{2}}du
\end{align*}
Shifting the contour to $\Re(u)=\max\{1,\sqrt{\sqrt{2}\pi x}\}$ then gives
\begin{align*}
H_{\iota_{\delta}}(x)&\ll \tfrac{1}{x^{3/4}}e^{-\sqrt{\sqrt{2}\pi x}}\\
&\ll \tfrac{1}{x}e^{-2\sqrt{x}}
\end{align*}
Where the implied constant is absolute. 
\end{proof}

\end{subsection}
\end{section}

\begin{section}{Poisson Summation}

With the notation of \eqref{L} the elliptic part of the trace formula (i.e. equation \eqref{2eq4}) is
\begin{equation*}
\sum_{\mp}\sum_{\substack{m\in \mathbb{Z}\\ m^2\pm 4p^k\neq \square}}\theta_{\infty}^{\mp}\left(\tfrac{m}{2p^{k/2}}\right)L(1,m^2\pm4p^k)
\end{equation*}

Our aim is to apply Poisson summation to the $m$-sum above. This, however, is not straightforward due to the problems caused by the singularities of $\theta_{\infty}^{\mp}$ and by the conditional convergence of the Dirichlet series defining the value of the $L$-functions. In the following paragraphs we will first review the problems and then state the simple but important observation, Proposition \ref{singprop1}, that will allow us to resolve these issues and apply Poisson summation.

\begin{subsection}{Remarks on Poisson Summation}\label{4.1}

As we said there are a few points to be resolved before Poisson summation can be applied. 
\begin{enumerate}

\item The $m$-sum is not running over the complete lattice (it is missing the values for which $m^2\pm4p^k=\square$) and adding these values manually is problematic since the $L$-functions, $L\left(s,\left(\tfrac{m^2\pm4p^k}{\cdot}\right)\right)$, have poles at $s=1$ when $m^2\pm4p^k=\square$.

\item The sums that define the values of the $L$-functions do not converge absolutely, hence the interchange of summations are problematic.

\item  The functions $\theta_{\infty}^{\mp}$ are not smooth. They have singularities of the prescribed type that we have discussed in \S\ref{secsing}.

\end{enumerate}

The first two of these problems are easily resolved by the introduction of the approximate functional equation which replaces the conditionally convergent series defining $L\left(1,\left(\tfrac{m^2\pm4p^k}{\cdot}\right)\right)$ with absolutely (and rapidly) converging sums. Substituting \eqref{afe} in \eqref{2eq4} results in
\begin{equation*}
\sum_{\mp}\sum_{\substack{m\in\mathbb{Z}\\ m^2\pm4p^k\neq \square}}\theta_{\infty}^{\mp}\left(\tfrac{m}{2p^{k/2}}\right)\sum_{\substack{f^2\mid m^2\pm4p^k}}^{'}\tfrac{1}{f}\sum_{l=1}^{\infty}\tfrac{1}{l}\left(\tfrac{(m^2\pm4p^k)/f^2}{l}\right)\left[F\left(\tfrac{lf^2}{A}\right)+\tfrac{l f^2}{\sqrt{|m^2\pm4p^k|}}H\left(\tfrac{lf^2A}{|m^2\pm4p^k|}\right)\right]\tag{$4'$}\label{elafe}
\end{equation*}
Where in order to not to complicate the notation we denoted $H_{\iota_{m^2\pm4p^k}}$ by $H$ keeping the dependence on $m$ and $p$ implicit. 

\end{subsection}

\begin{subsection}{Smoothing and Poisson Summation}

Although introducing the approximate functional equation resolves the first two problems it does not immediately resolve the third. The crucial observation, stated in the next proposition, is that by choosing the parameter $A$ appropriately we can smooth out the function $\theta_{\infty}^{\mp}$ which allows us to use Poisson summation without trouble.

 \begin{prop}\label{singprop1}Let $\alpha>0$ and $\Phi(x)\in \mathcal{S}(\mathbb{R})$ be a Scwartz class function. Then the functions $\theta_{\infty}^{\mp}(x)\Phi(|1-x^2|^{-\alpha})$ and $|1-x^2|^{-1/2}\theta_{\infty}^{\mp}(x)\Phi(|1-x^2|^{-\alpha})$ are both smooth.

\end{prop}

\begin{proof}
By \eqref{2starstar} we see that the only problematic points are $x=\pm1$. Without loss of generality we can take $x=1$ since the argument is the same for both points. Furthermore the argument is verbatim for both functions so without loss of generality we will treat the first function. We will show that both the left and right derivatives of the functions at $x=1$ are $0$, which will show that the function is differentiable. It will then be clear from the proof that the same argument shows that left and right derivatives of all orders exit and are $0$.

We begin with the left derivative. Consider the difference quotient,
\begin{align*}
\lim_{h\rightarrow0^+ }\tfrac{\theta_{\infty}^{\mp}\left(1-h\right)\Phi\left((1-(1-h)^2)^{-\alpha}\right)}{h}&=\lim_{h\rightarrow0^+ }\tfrac{\theta_{\infty}^{\mp}\left(1-h\right)\Phi\left((2h-h^2)^{-\alpha}\right)}{h}
\end{align*}

Since $\Phi$ is Schwarz class, for any $M>0$ we have 
\[\Phi(x)=O_M(x^{-M})\]
Therefore as $h\rightarrow0^+$
\[\Phi\left((2h-h^2)^{-\alpha}\right)=O_M\left((2h-h^2)^{M\alpha}\right)\]

Therefore,
\[\tfrac{\theta_{\infty}^{\mp}\left(1-h\right)\Phi\left((2h-h^2)^{-\alpha}\right)}{h}=O_M\left(\tfrac{\theta_{\infty}^{\mp}(1-h)}{h}(2h-h^2)^{M\alpha}\right)\]

By \eqref{2star} $\theta_{\infty}^{\mp}$ is bounded and hence we see that the limit is $0$. Now note that the same argument applies verbatim to the right derivative hence proves differentiability. Since $M$ was arbitrary the same argument proves that all the derivatives exists. 

\end{proof}

Recall that in \eqref{afe} the constant $A>0$ is yet to be chosen. By Proposition \ref{singprop1} and estimates in Lemma \ref{afelem2}, for any $1>\alpha>0$ if we choose $A=|m^2\pm4p^k|^{\alpha}$ then both 
\[\theta_{\infty}^{\mp}\left(\tfrac{m}{2p^{k/2}}\right)F\left(\tfrac{lf^2}{|m^2\pm4p^k|^{\alpha}}\right)\]
and
\[|m^2\pm4p^k|^{-1/2}\theta_{\infty}^{\mp}\left(\tfrac{m}{2p^{1/2}}\right)H\left(\tfrac{lf^2}{|m^2\pm4p^k|^{1-\alpha}}\right)\]
are smooth functions of the variable $m$, and hence Poisson summation can be applied.

\begin{thm}\label{pois}Let $1>\alpha>0$ and set $A=|m^2\pm 4p^k|^{\alpha}$ in \eqref{afe}. Then 
\begin{multline}\label{4eq13}
\eqref{2eq4}+\Sigma(\square)=\tfrac{p^{k/2}}{2}\sum_{\mp}\sum_{f=1}^{\infty}\tfrac{1}{f^3}\sum_{l=1}^{\infty}\tfrac{1}{l^2}\\
\sum_{\xi\in\mathbb{Z}}\left\{\int\theta_{\infty}^{\mp}\left(x\right)\left[F\left(\tfrac{lf^2(4p^k)^{-\alpha}}{|x^2\pm1|^{\alpha}}\right)+\tfrac{lf^2p^{-k/2}}{2\sqrt{|x^2\pm1|}}H\left(\tfrac{lf^2(4p^k)^{\alpha-1}}{|x^2\pm1|^{1-\alpha}}\right)\right]e\left(\tfrac{-x\xi p^{k/2}}{2 l f^2}\right) dx\right\} \cdot Kl_{l,f}(\xi,\mp p^k)
\end{multline}
Where\footnote{For $q\neq p$ a prime, $Kl_{q,1}(\xi,p^k)$ with $\gcd(\xi,q)=1$ is the classical Kloosterman sum $S(\bar{2}\xi,2\xi p^k;q)$ (cf. \cite{S} equation (70)), hence the notation.},

\begin{align*}
Kl_{l,f}(\xi,\mp p^{k}):&=\sum_{\substack{a\bmod 4lf^2\\ a^2\pm4p^k\equiv 0\bmod f^2\\ \tfrac{a^2\pm4p^k}{f^2}\equiv0,1\bmod 4}}\left(\tfrac{(a^2\pm4p^k)/f^2}{l}\right)e\left(\tfrac{a\xi}{4lf^2}\right)\\
H\left(\tfrac{lf^2(4p^k)^{\alpha-1}}{|x^2\pm1|^{1-\alpha}}\right)&=\begin{cases}H_0\left(\tfrac{lf^2(4p^k)^{\alpha-1}}{|x^2\pm1|^{1-\alpha}}\right)&\text{if $x^2\pm1>0$}\\ H_1\left(\tfrac{lf^2(4p^k)^{\alpha-1}}{|x^2\pm1|^{1-\alpha}}\right)& \text{if $x^2\pm1<0$}\end{cases}
\end{align*}
$H_{0,1}$ being defined in Corollary \ref{approxfuneq}, and
\begin{equation*}
\Sigma(\square):=\sum_{\mp}\sum_{\substack{m\in\mathbb{Z}\\ m^2\pm4p^k= \square}}\theta_{\infty}^{\mp}\left(\tfrac{m}{2p^{k/2}}\right)\sum_{\substack{f^2\mid m^2\pm4p^k}}^{'}\tfrac{1}{f}\sum_{l=1}^{\infty}\tfrac{1}{l}\left(\tfrac{(m^2\pm4p^k)/f^2}{l}\right)\left[F\left(\tfrac{lf^2}{A}\right)+\tfrac{l f^2}{\sqrt{|m^2\pm4p^k|}}H\left(\tfrac{lf^2A}{|m^2\pm4p^k|}\right)\right]
\end{equation*}

\end{thm}

\begin{proof} Since the $l$-sums in \eqref{elafe} converge absolutely we can add and subtract the values of $m\in \mathbb{Z}$ for which $m^2\pm4p^k=\square$ to \eqref{elafe}. Therefore $\eqref{2eq4}$ can be written as
\begin{equation*}
\sum_{\mp}\sum_{\substack{m\in\mathbb{Z}}}\theta_{\infty}^{\mp}\left(\tfrac{m}{2p^{k/2}}\right)\sum_{\substack{f^2\mid m^2\pm4p^k}}^{'}\tfrac{1}{f}\sum_{l=1}^{\infty}\tfrac{1}{l}\left(\tfrac{(m^2\pm4p^k)/f^2}{l}\right)\left[F\left(\tfrac{lf^2}{A}\right)+\tfrac{l f^2}{\sqrt{|m^2\pm4p^k|}}H\left(\tfrac{lf^2A}{|m^2\pm4p^k|}\right)\right]-\Sigma(\square)
\end{equation*}
The sum $\Sigma(\square)$ is the second term on the left iof \eqref{4eq13} and will not be analyzed any further. So from now on we focus on the first sum. Note the Kronecker symbol $\left(\tfrac{(m^2\pm4p^k)/f^2}{l}\right)$ as well as the condition that $\tfrac{m^2\pm4p^k}{f^2}\equiv 0,1\bmod 4$ are periodic (in $m$)  $\bmod 4lf^2$. Therefore by interchanging the $f$ and $l$-sums with the $m$-sum (which we can do because the $l$-sum converges absolutely and the $f$-sum is finite), and breaking the $m$-sum into arithmetic progressions $\bmod 4lf^2$ we can rewrite the first sum as follows:
\begin{multline*}
\sum_{\mp}\sum_{f=1}^{\infty}\tfrac{1}{f}\sum_{l=1}^{\infty}\tfrac{1}{l}\sum_{\substack{a\bmod 4lf^2\\ a^2\pm4p^k\equiv 0\bmod f^2\\ \tfrac{a^2\pm4p^k}{f^2}\equiv 0,1\bmod 4}}\left(\tfrac{(m^2\pm4p^k)/f^2}{l}\right)\\ 
\sum_{\substack{m\in\mathbb{Z}\\ m\equiv a\bmod 4lf^2}}\theta_{\infty}^{\mp}\left(\tfrac{m}{2p^{k/2}}\right)\left[F\left(\tfrac{lf^2}{|m^2\pm 4p^k|^{\alpha}}\right)+\tfrac{ lf^2}{\sqrt{|m^2\pm4p^k|}}H\left(\tfrac{lf^2}{|m^2\pm4p^k|^{1-\alpha}}\right)\right]
\end{multline*}

 Applying Poisson summation to the $m$-sum (which we can by Proposition \ref{singprop1}, i.e. see the argument prior to the statement of the theorem) proves the theorem.

\end{proof}

\end{subsection}
\end{section}

\begin{section}{An Auxiliary Dirichlet Series}

For any $n\in \mathbb{Z}$ and $z\in \mathbb{C}$, let $D(z;n)$ be defined by

\begin{equation}\label{5eq14}
D(z;n):=\sum_{f=1}^{\infty}\tfrac{1}{f^{2z+1}}\sum_{l=1}^{\infty}\tfrac{Kl_{l,f}(0,n)}{l^{z+1}}
\end{equation}

In order to analyze the $\xi=0$ term of the sum in Theorem \ref{pois} we will need the analytic properties of $D(z;n)$.

\begin{lemma}\label{pstlem6}
\[D(z;n)=\prod_{p}D_p(z;n)\]
Where for each prime $p$, $D_p(z;n)$ is defined by
\[D_{p}(z;n):=\sum_{u=0}^{\infty}\tfrac{1}{p^{u(2z+1)}}\sum_{v=0}^{\infty}\tfrac{Kl_{p^v,p^u}(0,n)}{p^{v(z+1)}}\]
\end{lemma}

\begin{proof}Chinese remainder theorem.

\end{proof}

\begin{lemma}\label{pstlem7}Let $p\nmid n$ be an prime. Then,
\[D_{p}(z;n)=\begin{cases}\tfrac{\left(1-\tfrac{1}{p^{z+1}}\right)}{\left(1-\tfrac{1}{p^{2z}}\right)}&\text{if $p$ is odd}\\ 4\tfrac{\left(1-\tfrac{1}{2^{z+1}}\right)}{\left(1-\tfrac{1}{2^{2z}}\right)}&\text{if $p=2$}\end{cases}\]

\end{lemma}

\begin{proof}Let us first assume that $p\equiv1\bmod2$. In order to compute $D_p(z;n)$ we need to compute $Kl_{p^v,p^u}(0,n)$ for various values of $u$ and $v$.
\begin{itemize}
\item $v=u=0$.

In this case the $Kl_{1,1}(0,n)$ is obviously $1$.

\item $v>0,\,u=0$. 
\begin{align*}
Kl_{p^v,1}(0,n)&=\sum_{a\bmod p^v}\left(\tfrac{a^2-4n}{p^v}\right)\\
&=\sum_{a_0\bmod p}\left(\tfrac{a_0^2-4n}{p^v}\right)\sum_{a_1\bmod p^{v-1}}1\\
&=p^{v-1}\sum_{a_0\bmod p}\left(\tfrac{a_0^2-4n}{p^v}\right)\tag{$i$}\label{pstlem7i}
\end{align*}
This last sum depends on the parity of $v$.

\begin{itemize}
\item $v\equiv 0\bmod 2$. 

In this case

\begin{align*}
\eqref{pstlem7i}&=p^{v-1}\sum_{\substack{a_0\bmod p\\ a_0^2\neq 4n}}1\\
&=p^v-p^{v-1}\left(1+\left(\tfrac{n}{p}\right)\right)\tag{$ii$}\label{pstlem7ii}
\end{align*}

\item $v\equiv 1 \bmod 2$.

In this case
\begin{equation*}
\eqref{pstlem7i}=-p^{v-1}\tag{$iii$}\label{pstlem7iii}
\end{equation*}
Where we used Lemma 2 of Appendix A of \cite{L1}.
\end{itemize}

\item $v=0,\, u>0$. 

Since $p\neq2$, 
\begin{align*}
Kl_{1,p^u}(0,n)&=\sum_{\substack{a\bmod p^{2u}\\ a^2\equiv 4n \bmod p^{2u}}}1\\
&=1+\left(\tfrac{n}{p}\right)\tag{$iv$}\label{pstlem7iv}
\end{align*}

\item $v,u>0$.
First of all the sum is clearly $0$ unless $n$ is a square $\bmod p$. If $n$ is a square $\bmod p$, since $p\neq2$, then there are two exactly square-roots of $4n$ modulo $p^{2u}$. Let us denote them by $n_1,n_2$. Then,
\begin{align*}
Kl_{p^v,p^u}(0,n)&=\sum_{\substack{a\bmod p^{v+2u}\\ a^2\equiv 4n\bmod p^{2u}}}\left(\tfrac{(a^2-4n)/p^{2u}}{p^v}\right)\\
&=p^{v-1}\sum_{\substack{a_0\bmod p^{1+2u}\\ a_0^2\equiv 4n\bmod p^{2u}}}\left(\tfrac{(a_0^2-4n)/p^{2u}}{p^v}\right)\\
&=p^{v-1}\sum_{\substack{a_0\bmod p^{1+2u}\\ a_0\equiv n_j\bmod p^{2u}}}\left(\tfrac{(a_0^2-4n)/p^{2u}}{p^v}\right)\\
&=p^{v-1}\sum_{\substack{a_2\bmod p\\ j=1,2}}\left(\tfrac{a_2n_j}{p^v}\right)\\
&=p^{v-1}\begin{cases}0&\text{if $v\equiv 1\bmod 2$}\\ (p-1)\left(1+\left(\tfrac{n}{p}\right)\right) &\text{if $v\equiv0\bmod 2$}\end{cases}\tag{$v$}\label{pstlem7v}
\end{align*}

\end{itemize}

We can now compute $D_p(z;n)$. 
\begin{align*}
D_p(z;n)&=\sum_{u=0}^{\infty}\tfrac{1}{p^{u(2z+1)}}\sum_{v=0}^{\infty}\tfrac{Kl_{p^{v},p^u}(0,n)}{p^{v(z+1)}}\\
&=1+\sum_{v=1}^{\infty}\tfrac{Kl_{p^v,1}(0,n)}{p^{v(z+1)}}+\sum_{u=1}^{\infty}\tfrac{Kl_{1,p^u}(0,n)}{p^{u(2z+1)}}+\sum_{u=1}^{\infty}\tfrac{1}{p^{u(2z+1)}}\sum_{v=1}^{\infty}\tfrac{Kl_{p^{v},p^u}(0,n)}{p^{v(z+1)}}
\end{align*}
Using \eqref{pstlem7i} to \eqref{pstlem7v}, we then get:
\begin{align*}
D_p(z;n)&=1+\sum_{v=1}^{\infty}\tfrac{Kl_{p^v,1}(0,n)}{p^{v(z+1)}}+\sum_{u=1}^{\infty}\tfrac{Kl_{1,p^u}(0,n)}{p^{u(2z+1)}}+\sum_{u=1}^{\infty}\tfrac{1}{p^{u(2z+1)}}\sum_{v=1}^{\infty}\tfrac{Kl_{p^{v},p^u}(0,n)}{p^{v(z+1)}}\\
&=1+\sum_{v=1}^{\infty}\tfrac{Kl_{p^{2v},1}(0,n)}{p^{2v(z+1)}}+\sum_{v=0}^{\infty}\tfrac{Kl_{p^{2v+1},1}(0,n)}{p^{(2v+1)(z+1)}}+\sum_{u=1}^{\infty}\tfrac{Kl_{1,p^u}(0,n)}{p^{u(2z+1)}}+\sum_{u=1}^{\infty}\tfrac{1}{p^{u(2z+1)}}\sum_{v=1}^{\infty}\tfrac{Kl_{p^{2v},p^u}(0,n)}{p^{2v(z+1)}}\\
&=1+\left(1-\tfrac{1}{p}\left(1+\left(\tfrac{n}{p}\right)\right)\right)\sum_{v=1}^{\infty}\tfrac{p^{2v}}{p^{2v(z+1)}}-\tfrac{1}{p}\sum_{v=0}^{\infty}\tfrac{p^{2v+1}}{p^{(2v+1)(z+1)}}+\left(1+\left(\tfrac{n}{p}\right)\right)\sum_{u=1}^{\infty}\tfrac{1}{p^{u(2z+1)}}\\
&\hspace{0.2in}+\left(1-\tfrac{1}{p}\right)\left(1+\left(\tfrac{n}{p}\right)\right)\sum_{u=1}^{\infty}\tfrac{1}{p^{u(2z+1)}}\sum_{v=1}^{\infty}\tfrac{p^{2v}}{p^{2v(z+1)}}\\
&=1+\left(1-\tfrac{1}{p}\left(1+\left(\tfrac{n}{p}\right)\right)\right)\sum_{v=1}^{\infty}\tfrac{1}{p^{2vz}}-\tfrac{1}{p}\sum_{v=0}^{\infty}\tfrac{1}{p^{(2v+1)z}}+\left(1+\left(\tfrac{n}{p}\right)\right)\sum_{u=1}^{\infty}\tfrac{1}{p^{u(2z+1)}}\\
&\hspace{0.2in}+\left(1-\tfrac{1}{p}\right)\left(1+\left(\tfrac{n}{p}\right)\right)\sum_{u=1}^{\infty}\tfrac{1}{p^{u(2z+1)}}\sum_{v=1}^{\infty}\tfrac{1}{p^{2vz}}\\
&=\tfrac{1-\frac{1}{p^{z+1}}}{1-\frac{1}{p^{2z}}}+\left(1+\left(\tfrac{n}{p}\right)\right)\sum_{u=1}^{\infty}\tfrac{1}{p^{u(2z+1)}}\sum_{v=0}^{\infty}\tfrac{1}{p^{2vz}}-\tfrac1p\left(1+\left(\tfrac{n}{p}\right)\right)\sum_{u=0}^{\infty}\tfrac{1}{p^{u(2z+1)}}\sum_{v=1}^{\infty}\tfrac{1}{p^{2vz}}\\
&=\tfrac{1-\frac{1}{p^{z+1}}}{1-\frac{1}{p^{2z}}}
\end{align*}

This finishes the proof of the lemma when $p\equiv1\bmod2$. The computation for $p=2$ follows the same argument using the properties of the Kronecker symbol $\left(\tfrac{\cdot}{2}\right)$. The only difference is that we need to do a case by case calculation depending on the congruence class of $n\bmod 8 $. We leave the details to the reader.
\end{proof}

\begin{lemma}\label{pstlem8}Let $p\mid n$ be a prime, and let $v_p(n)$ denote the $p$-adic valuation of $n$. Then,
\[D_p(z;n)=\begin{cases}\tfrac{\left(1-\frac{1}{p^{z(v_p(n)+1)}}\right)\left(1-\frac{1}{p^{z+1}}\right)}{\left(1-\frac{1}{p^{2z}}\right)\left(1-\tfrac{1}{p^{z}}\right)}&\text{if $p$ is odd}\\ 4\tfrac{\left(1-\frac{1}{2^{z(v_p(n)+1)}}\right)\left(1-\frac{1}{2^{z+1}}\right)}{\left(1-\frac{1}{2^{2z}}\right)\left(1-\tfrac{1}{2^{z}}\right)}&\text{if $p=2$}\end{cases}\]

\end{lemma}

\begin{proof}As in the proof of Lemma \ref{pstlem7} we first assume that $p\equiv1\bmod 2$. The computation depends on the parity of the $p$-adic valuation of $n$. Let $r=v_p(n)$ throughout the proof. As in the proof of Lemma \ref{pstlem7} we need to compute the values of $Kl_{p^v,p^u}(0,n)$ first.

\begin{itemize}
\item $r\equiv 1\bmod 2$.
We divide the computation into cases depending on the values of $v$ and $u$.
\begin{itemize}

\item $v=u=0$.

In this case $Kl_{p^v,p^u}(0,n)$ is obviously $1$.

\item $v>0,\,u=0$.

Since $p\mid n$,

\begin{align*}
Kl_{p^v,1}(0,n)&=\sum_{a\bmod p^v}\left(\tfrac{a^2-4n}{p^v}\right)\\
&=\sum_{a\bmod p^v}\left(\tfrac{a^2}{p^v}\right)\\
&=p^{v}-p^{v-1}\tag{$i$}\label{pstlem8i}
\end{align*}

\item $2u>r$.

In this case we use the assumption that $r\equiv1\bmod 2$ and that $p\neq2$. Note that in this case the sum runs over $a\bmod p^{2u}$ such that $a^2\equiv 4n\bmod p^{2u}$. If $a\bmod p^{2u}$ is such that $a^2\equiv 4n \bmod p^{2u}$ then we need to have $a=p^{\frac{r+1}{2}}a_0$ for some $a_0\bmod p^{2u-\frac{r+1}{2}}$. But in this case $a^2=p^{r+1}a_0^2\equiv 0\bmod p^{r+1}$, therefore we cannot have $a^2\equiv 4n\bmod p^{2u}$. Hence in this case the sum is $0$.

\item $r>2u>0,\,v>0$.

In this case $n\equiv 0\bmod p^{2u}$ and hence in order to have $a^2\equiv a\bmod p^{2u}$ we necessarily have $a\equiv 0\bmod p^u$. Then the sum is,
\begin{align*}
Kl_{p^{v},p^{u}}(0,n)&=\sum_{\substack{a\bmod p^{v+2u}\\ a^2\equiv 4n\bmod p^{2u}}}\left(\tfrac{(a^2-4n)/p^{2u}}{p^v}\right)\\
&=\sum_{a_0\bmod p^{v+u}}\left(\tfrac{a_0^2}{p^v}\right)\tag{$\star$}\label{pstlem8star}\\
&=p^{u+v}-p^{u+v-1}\tag{$ii$}\label{pstlem8ii}
\end{align*}
Where in passing to \eqref{pstlem8star} we used the assumption that $r\equiv1\bmod 2$ so that $a/p^{2u}\equiv 0\bmod p$.

\item $r>2u>0,\,v=0$.

In this case,

\begin{align*}
Kl_{1,p^{u}}(0,n)&=\sum_{\substack{a\bmod p^{2u}\\ a^2\equiv 4n\bmod p^{2u}}}1\\
&=p^u\tag{$iii$}\label{pstlem8iii}
\end{align*}

\end{itemize}

Using \eqref{pstlem8i}, \eqref{pstlem8ii}, \eqref{pstlem8iii} and the argument above we see that,
\begin{align*}
D_p(z;n)&=1+\sum_{v=1}^{\infty}\tfrac{Kl_{p^v,1}(0,n)}{p^{v(z+1)}}+\sum_{u=1}^{\infty}\tfrac{Kl_{1,p^u}(0,n)}{p^{u(2z+1)}}+\sum_{u=1}^{\infty}\tfrac{1}{p^{u(2z+1)}}\sum_{v=1}^{\infty}\tfrac{Kl_{p^{v},p^u}(0,n)}{p^{v(z+1)}}\\
&=1+\left(1-\tfrac{1}{p}\right)\sum_{v=1}^{\infty}\tfrac{p^{v}}{p^{v(z+1)}}+\sum_{u=1}^{\frac{r-1}{2}}\tfrac{p^u}{p^{u(2z+1)}}+\left(1-\tfrac1p\right)\sum_{u=1}^{\frac{r-1}{2}}\tfrac{p^u}{p^{u(2z+1)}}\sum_{v=1}^{\infty}\tfrac{p^v}{p^{v(z+1)}}\\
&=\left(1-\tfrac1p\right)\sum_{u=0}^{\frac{r-1}{2}}\tfrac{1}{p^{2uz}}\sum_{v=1}^{\infty}\tfrac{1}{p^{vz}}+\sum_{u=0}^{\frac{r-1}{2}}\tfrac{1}{p^{2uz}}\\
&=\tfrac{1-\frac{1}{p^{z+1}}}{1-\frac{1}{p^z}}\sum_{u=0}^{\frac{r-1}{2}}\tfrac{1}{p^{2uz}}\\
&=\tfrac{\left(1-\frac{1}{p^{z+1}}\right)\left(1-\frac{1}{p^{z(r+1)}}\right)}{\left(1-\frac{1}{p^z}\right)\left(1-\frac{1}{p^{2z}}\right)}
\end{align*}
Which finishes the proof in the case of $r\equiv 1\bmod 2$.

\item $r\equiv 0\bmod 2$. 

Let $r=2r_0$. We proceed as before and first compute $Kl_{p^{v},p^u}(0,n)$. The computation once again depends on the values of $v$ and $u$.

\begin{itemize}

\item $v=u=0$.

Once again $Kl_{1,1}(0,n)=1$.

\item $v>0,\,u=0$.

In this case the result of \eqref{pstlem8i} is still valid.

\item $r_0\geq u>0,\, v=0$.

In this case \eqref{pstlem8iii} is still valid.

\item $u> r_0,\, v=0$.

In this case we need to compute,

\begin{equation*}
Kl_{1,p^u}(0,n)=\sum_{\substack{a\bmod p^{2u}\\ a^2\equiv4n\bmod p^{2u}}}1
\end{equation*}
Let $n=p^{2r_0}n_0$. Then the sum vanishes unless $\left(\tfrac{n_0}{p}\right)=1$. If this is the case then we have,
\begin{align*}
Kl_{1,p^u}(0,n)&=\sum_{\substack{a\bmod p^{2u}\\ a^2\equiv4n\bmod p^{2u}}}1\\
&=\sum_{\substack{a_0\bmod p^{2u-r_0}\\ a_0^2\equiv4n_0\bmod p^{2u-2r_0}}}1\\
&=2p^{r_0}
\end{align*}
Therefore in this case,
\[Kl_{1,p^u}(0,n)=\left(1+\left(\tfrac{n_0}{p}\right)\right)p^{r_0}\tag{$iv$}\label{pstlem8iv}\]

\item $r_0>u>0,\,v>0$.

In this case \eqref{pstlem8ii} is still valid.

\item $u=r_0,\, v>0$.

In this case we have,
\begin{align*}
Kl_{p^{v},p^{u}}(0,n)&=\sum_{\substack{a\bmod p^{v+2u}\\ a^2\equiv 4n\bmod p^{2u}}}\left(\tfrac{(a^2-4n)/p^{2u}}{p^v}\right)\\
&=\sum_{a_0\bmod p^{v+u}}\left(\tfrac{a_0^2-4n_0}{p^v}\right)\\
&=\begin{cases}-p^{v+u-1}&\text{if $v\equiv 1\bmod 2$}\\ p^{v+u}-p^{v+u-1}\left(1+\left(\frac{n_0}{p}\right)\right)&\text{if $v\equiv 0\bmod 2$}\end{cases}
\end{align*}
Where we used Lemma 2 of Appendix A of \cite{L1} in the first line.
\item $u>r_0, \, v>0$. 

In this case we need to compute,

\begin{equation*}
Kl_{p^v,p^u}(0,n)=\sum_{\substack{a\bmod p^{v+2u}\\ a^2\equiv4n\bmod p^{2u}}}\left(\tfrac{(a^2-4n)/p^{2u}}{p^v}\right)
\end{equation*}
Let $n=p^{2r_0}n_0$, where $v_p(n_0)=0$. Then, since $p\neq2$, in order to have $a^2\equiv 4n\bmod p^{2u}$ we need to have $\left(\tfrac{n_0}{p}\right)=1$, i.e. $n_0$ is a square modulo $p$. This, by Hensel's lemma, implies that $n_0$ is a square modulo $p^{2u-2r_0+1}$. Let us assume that this is the case and denote the square-roots (which there are exactly two since $p\neq2$) of $n_0$ modulo $p^{2u-2r_0+1}$ by $u_1,\,u_2$, i.e. $u_j^2\equiv n_0\bmod p^{2u-2r_0+1}$. Then $Kl_{p^v,p^u}(0,n)$ can be written as,
\begin{align*}
Kl_{p^v,p^u}(0,n)&=\sum_{\substack{a\bmod p^{v+2u}\\ a^2\equiv4n\bmod p^{2u}}}\left(\tfrac{(a^2-4n)/p^{2u}}{p^v}\right)\\
&=p^{v-1}\sum_{\substack{a_0\bmod p^{1+2u}\\ a_0^2\equiv4n\bmod p^{2u}}}\left(\tfrac{(a_0^2-4n)/p^{2u}}{p^v}\right)\\
&=p^{v-1}\sum_{\substack{a_1\bmod p^{1+2u-r_0}\\ a_1^2\equiv4n_0\bmod p^{2(u-r_0)}}}\left(\tfrac{(a_1^2-4n_0)/p^{2(u-r_0)}}{p^v}\right)\\
&=p^{v+r_0-1}\sum_{j=1,2}\sum_{\substack{a_3\bmod p}}\left(\tfrac{a_3u_j}{p^v}\right)\\
&=p^{v+r_0-1}2\begin{cases}0&\text{if $v\equiv 1\bmod 2$}\\ p-1&\text{if $v\equiv 0\bmod 2$}\end{cases}
\end{align*}
We therefore get,
\[Kl_{p^v,p^u}(0,n)=p^{v+r_0}\left(1-\tfrac{1}{p}\right)\left(1+\left(\tfrac{n_0}{p}\right)\right)\begin{cases}0&\text{if $v\equiv 1\bmod 2$}\\ 1&\text{if $v\equiv 0\bmod 2$}\end{cases}\tag{$v$}\label{pstlem8v}\]

\end{itemize}

We can now compute $D_p(z;n)$. By using \eqref{pstlem8i} to \eqref{pstlem8v} we get,

\begin{align*}
D_p(z;n)&=1+\sum_{v=1}^{\infty}\tfrac{Kl_{p^v,1}(0,n)}{p^{v(z+1)}}+\sum_{u=1}^{r_0}\tfrac{Kl_{1,p^u}(0,n)}{p^{u(2z+1)}}+\sum_{u=r_0+1}^{\infty}\tfrac{Kl_{1,p^u}(0,n)}{p^{u(2z+1)}}\\
&\hspace{0.5in}+\sum_{u=1}^{r_0-1}\tfrac{1}{p^{u(2z+1)}}\sum_{v=1}^{\infty}\tfrac{Kl_{p^{v},p^u}(0,n)}{p^{v(z+1)}}+\tfrac{1}{p^{r_0(2z+1)}}\sum_{v=1}^{\infty}\tfrac{Kl_{p^v,p^{r_0}}(0,n)}{p^{v(z+1)}}\\
&\hspace{1in}+\sum_{u=r_0+1}^{\infty}\tfrac{1}{p^{u(2z+1)}}\sum_{v=1}^{\infty}\tfrac{Kl_{p^{v},p^u}(0,n)}{p^{v(z+1)}}\\
&=1+\left(1-\tfrac1p\right)\sum_{v=1}^{\infty}\tfrac{1}{p^{vz}}+\sum_{u=1}^{r_0}\tfrac{1}{p^{2uz}}+p^{r_0}\left(1+\left(\tfrac{n_0}{p}\right)\right)\sum_{u=r_0+1}^{\infty}\tfrac{1}{p^{u(2z+1)}}\\
&+\left(1-\tfrac1p\right)\sum_{u=1}^{r_0-1}\tfrac{1}{p^{2uz}}\sum_{v=1}^{\infty}\tfrac{1}{p^{vz}}+\tfrac{p^{r_0}}{p^{r_0(2z+1)}}\left(1-\tfrac1p\left(1+\left(\tfrac{n_0}{p}\right)\right)\right)\sum_{v=1}^{\infty}\tfrac{1}{p^{2vz}}\\
&-\tfrac{p^{r_0-1}}{p^{r_0(2z+1)}}\sum_{v=0}^{\infty}\tfrac{1}{p^{(2v+1)z}}+p^{r_0}\left(1-\tfrac{1}{p}\right)\left(1+\left(\tfrac{n_0}{p}\right)\right)\sum_{u=r_0+1}^{\infty}\tfrac{1}{p^{u(2z+1)}}\sum_{v=1}^{\infty}\tfrac{1}{p^{2vz}}
\end{align*}

\begin{align*}
&=1+\left(1-\tfrac1p\right)\sum_{v=1}^{\infty}\tfrac{1}{p^{vz}}+\sum_{u=1}^{r_0}\tfrac{1}{p^{2uz}}+\left(1-\tfrac1p\right)\sum_{u=1}^{r_0-1}\tfrac{1}{p^{2uz}}\sum_{v=1}^{\infty}\tfrac{1}{p^{vz}}\\
&\hspace{1in}+\tfrac{p^{r_0}}{p^{r_0(2z+1)}}\sum_{v=1}^{\infty}\tfrac{1}{p^{2vz}}-\tfrac{p^{r_0-1}}{p^{r_0(2z+1)}}\sum_{v=0}^{\infty}\tfrac{1}{p^{(2v+1)z}}\\
&=1+\left(1-\tfrac1p\right)\sum_{u=0}^{r_0-1}\tfrac{1}{p^{2uz}}\sum_{v=1}^{\infty}\tfrac{1}{p^{vz}}+\sum_{u=1}^{r_0}\tfrac{1}{p^{2uz}}+\tfrac{\left(\frac{1}{p^z}-\frac{1}{p}\right)}{p^{(2r_0+1)z}\left(1-\frac{1}{p^{2z}}\right)}\\
&=\tfrac{\left(1-\frac{1}{p^{2(r_0+1)z}}\right)}{\left(1-\frac{1}{p^{2z}}\right)}+\tfrac{\left(1-\frac1p\right)\left(1-\frac{1}{p^{2r_0z}}\right)}{p^z\left(1-\frac{1}{p^{2z}}\right)\left(1-\frac{1}{p^z}\right)}+\tfrac{\left(\frac{1}{p^z}-\frac{1}{p}\right)}{p^{(2r_0+1)z}\left(1-\frac{1}{p^{2z}}\right)}\\
&=1+\sum_{u=0}^{r_0-1}\tfrac{1}{p^{2uz}}\sum_{v=1}^{\infty}\tfrac{1}{p^{vz}}+\sum_{u=1}^{r_0}\tfrac{1}{p^{2uz}}-\tfrac{1}{p}\sum_{u=0}^{r_0-1}\tfrac{1}{p^{2uz}}\sum_{v=1}^{\infty}\tfrac{1}{p^{vz}}-\tfrac{1}{p^{(2r_0+1)z+1}\left(1-\frac{1}{p^{2z}}\right)}\\
&=1+\tfrac{\left(1-\frac{1}{p^{2r_0z}}\right)}{p^{2z}\left(1-\frac{1}{p^{z}}\right)\left(1-\frac{1}{p^{2z}}\right)}-\tfrac{\left(1-\frac{1}{p^{(2r_0+1)z}}\right)}{p^{z+1}\left(1-\frac{1}{p^z}\right)\left(1-\frac{1}{p^{2z}}\right)}\\
&=\tfrac{\left(1-\frac{1}{p^{(2r_0+1)z}}\right)\left(1-\frac{1}{p^{z+1}}\right)}{\left(1-\frac{1}{p^z}\right)\left(1-\frac{1}{p^{2z}}\right)}
\end{align*}
This finishes the proof of the lemma for the case $p\equiv1\bmod 2$. The calculations for $p=2$ follow the same lines and as in the proof of Lemma \ref{pstlem7} we leave this case to the reader.
\end{itemize}

\end{proof}

\begin{cor}\label{pstcor9}Let $n\in\mathbb{Z}$. Then,
\[D(z;n)=4\tfrac{\zeta(2z)}{\zeta(z+1)}\prod_{p\mid n}\tfrac{\left(1-p^{-z(v_p(n)+1)}\right)}{\left(1-p^{-z}\right)}\]

\end{cor}

\begin{proof}Follows from lemmas \ref{pstlem7} and \ref{pstlem8}.

\end{proof}

\end{section}

\begin{section}{Isolation of the Contribution of Special Representations}

In this section we will isolate the special representations as promised in the introduction and finish the proof of Theorem \ref{mainthm}. We will identify the contribution of the trivial representation and the residues of Eisenstein series (denoted by $\xi_0$ in \S\ref{obstacles}) to the trace formula in the dominant term, $\eqref{4eq13}_{\xi=0}$, of \eqref{4eq13}. Where we define $\eqref{4eq13}_{\xi=0}$ by 
\begin{equation*}
\eqref{4eq13}_{\xi=0}:=\tfrac{p^{k/2}}{2}\sum_{\mp}\sum_{f=1}^{\infty}\tfrac{1}{f^3}\sum_{l=1}^{\infty}\tfrac{Kl_{l,f}(0,\mp p^k)}{l^2}
\left\{\int\theta_{\infty}^{\mp}\left(x\right)\left[F\left(\tfrac{lf^2(4p^k)^{-\alpha}}{|x^2\pm1|^{\alpha}}\right)+\tfrac{lf^2p^{-k/2}}{2\sqrt{|x^2\pm1|}}H\left(\tfrac{lf^2(4p^k)^{\alpha-1}}{|x^2\pm1|^{1-\alpha}}\right)\right] dx\right\} 
\end{equation*}

\begin{thm}\label{6thm1}Let $1>\alpha>0$, and $\upsilon>0$ be any number such that $\zeta(u+1)$ does not have any zeros for $|u|<\upsilon$ (Such an $\upsilon$ exists since $\zeta(u+1)$ is non-zero at $u=0$ and the zeta function is meropmorphic.). Let
$\mathcal{C}_{\upsilon}=\begin{setdef}{(0,it)}{t\in (-\infty,-\upsilon)\cup(\upsilon,\infty)}\end{setdef}\cup C_{\upsilon}$, and $C_{\upsilon}$ denotes the left-half of the circle of radius $\upsilon$ around $0$. Then,
\begin{align*}
\eqref{4eq13}_{\xi=0}&=2p^{k/2}\tfrac{\left(1-p^{-(k+1)}\right)}{\left(1-p^{-1}\right)}\sum_{\mp}\int\theta_{\infty}^{\mp}\left(x\right)dx-(k+1)\sum_{\mp}\int_{x^2\pm1>0}\tfrac{\theta_{\infty}^{\mp}(x)}{\sqrt{|x^2\pm1|}}dx\\
&+\tfrac{p^{k/2}}{2}\sum_{\mp}\int\theta_{\infty}^{\mp}\left(x\right)\left[4\tfrac{1}{2\pi i}\int_{(-1)}\tilde{F}(u)\left(\tfrac{(4p^k)^{-\alpha}}{|x^2\pm1|^{\alpha}}\right)^{-u}\tfrac{\zeta(2u+2)}{\zeta(u+2)}\tfrac{\left(1-p^{-(u+1)(k+1)}\right)}{\left(1-p^{-(u+1)}\right)}du\right.\\
&\left.\hspace{1.1in}+\tfrac{2\sqrt{\pi}p^{-k/2}}{\sqrt{|x^2\pm1|}}\tfrac{1}{2\pi i}\int_{\mathcal{C}_{\upsilon}}\tilde{F}(u)\tfrac{\Gamma\left(\frac{\iota_{(x^2\pm1)}+u}{2}\right)}{\Gamma\left(\frac{\iota_{(x^2\pm1)}+1-u}{2}\right)}\left(\tfrac{\pi (4p^k)^{\alpha-1}}{|x^2\pm1|^{1-\alpha}}\right)^{-u}\tfrac{\zeta(2u)}{\zeta(u+1)}\tfrac{\left(1-p^{-u(k+1)}\right)}{\left(1-p^{-u}\right)}du\right]dx
\end{align*}

Where $\iota_{(x^2\pm1)}=0,1$ depending on $x^2\pm1>0$ or $<0$ respectively (as already defined in \eqref{sharp}). 
\end{thm}

\begin{proof}
The $\xi=0$ term in \eqref{4eq13} is
\[\tfrac{p^{k/2}}{2}\sum_{\mp}\sum_{f=1}^{\infty}\tfrac{1}{f^3}\sum_{l=1}^{\infty}\tfrac{1}{l^2}
\left\{\int\theta_{\infty}^{\mp}\left(x\right)\left[F\left(\tfrac{lf^2(4p^k)^{-\alpha}}{|x^2\pm1|^{\alpha}}\right)+\tfrac{lf^2p^{-k/2}}{2\sqrt{|x^2\pm1|}}H\left(\tfrac{lf^2(4p^k)^{\alpha-1}}{|x^2\pm1|^{1-\alpha}}\right)\right]dx\right\} \cdot Kl_{l,f}(0,\mp p^k)\tag{$\circ$}\label{circ}\]
Where $H=H_0$ if $x^2\pm1>0$ and $H=H_1$ if $x^2\pm1<0$ (cf. Theorem \ref{pois}). Let $\tilde{F}$ denote the Mellin transform of $F$. By Lemma \ref{lemmaF}, $\tilde{F}(z)$ is holomorphic for $\Re(z)>0$. Therefore by Mellin inversion we have
\[F(y)=\tfrac{1}{2\pi i }\int_{(1)}\tilde{F}(u)y^{-u}du\]
Also recall that 
\begin{align*}
H_{0}(y)&=\tfrac{\sqrt{\pi}}{2\pi i}\int_{(1)}\tfrac{\Gamma\left(\frac{u}{2}\right)}{\Gamma\left(\frac{1-u}{2}\right)}(\pi y)^{-u}\tilde{F}\left(u\right)du\\
H_{1}(y)&=\tfrac{\sqrt{\pi}}{2\pi i}\int_{(1)}\tfrac{\Gamma\left(\frac{1+u}{2}\right)}{\Gamma\left(\frac{2-u}{2}\right)}(\pi y)^{-u}\tilde{F}\left(u\right)du
\end{align*}
We will need to distinguish into cases according to $x^2\pm1<0$ or not. In the first case we have $H=H_1$ and in the second $H=H_0$. We also note that when the sign in the first sum in \eqref{circ} is $-$ we necessarily have $x^2+1>0$. 

\begin{itemize}

\item $x^2\pm1<0$. 

As we have noted above, in this case we necessarily have the sign $+$ in the first sum of \eqref{circ}. Therefore we have,

\begin{multline*}
\tfrac{p^{k/2}}{2}\sum_{f=1}^{\infty}\tfrac{1}{f^3}\sum_{l=1}^{\infty}\tfrac{Kl_{l,f}(0, p^k)}{l^2}\\
\left\{\int_{|x|<1}\theta_{\infty}^{+}\left(x\right)\tfrac{1}{2\pi i}\int_{(1)}\tilde{F}(u)\left[\left(\tfrac{lf^2(4p^k)^{-\alpha}}{(1-x^2)^{\alpha}}\right)^{-u}+\tfrac{\sqrt{\pi}lf^2p^{-k/2}}{2\sqrt{1-x^2}}\tfrac{\Gamma\left(\frac{1+u}{2}\right)}{\Gamma\left(\frac{2-u}{2}\right)}\left(\tfrac{\pi lf^2(4p^k)^{\alpha-1}}{(1-x^2)^{1-\alpha}}\right)^{-u}\right]dudx\right\} \tag{$\circ_1$}\label{circ'}
\end{multline*}
Note that the integrand in the $u$-integral is holomorphic for $\Re(u)>0$ therefore we can move the u contour to right without changing the value of the integral. Then, by moving the contour to $\Re(u)=c>1$ and using the trivial bound $|Kl_{l,f}(0,n)|<4lf^2$ we can ensure that the $l$ and $f$-sums and the integrals converge absolutely and bring the sums into the integrals and get
\begin{multline*}
\tfrac{p^{k/2}}{2}\int_{|x|<1}\theta_{\infty}^{+}\left(x\right)\left[\tfrac{1}{2\pi i}\int_{(c)}\tilde{F}(u)\left(\tfrac{(4p^k)^{-\alpha}}{(1-x^2)^{\alpha}}\right)^{-u}D(u+1;p^k)du\right.\\
\left.+\tfrac{\sqrt{\pi}p^{-k/2}}{2\sqrt{1-x^2}}\tfrac{1}{2\pi i}\int_{(c)}\tilde{F}(u)\tfrac{\Gamma\left(\frac{1+u}{2}\right)}{\Gamma\left(\frac{2-u}{2}\right)}\left(\tfrac{\pi (4p^k)^{\alpha-1}}{(1-x^2)^{1-\alpha}}\right)^{-u}D(u;p^k)du\right]dx
\end{multline*}
Where $D(u;p^k)$ is as in \eqref{5eq14}. Using Corollary \ref{pstcor9} we see that this is equal to 
\begin{multline*}
\tfrac{p^{k/2}}{2}\int_{|x|<1}\theta_{\infty}^{+}\left(x\right)\left[4\tfrac{1}{2\pi i}\int_{(c)}\tilde{F}(u)\left(\tfrac{(4p^k)^{-\alpha}}{(1-x^2)^{\alpha}}\right)^{-u}\tfrac{\zeta(2u+2)}{\zeta(u+2)}\tfrac{\left(1-p^{-(u+1)(k+1)}\right)}{\left(1-p^{-(u+1)}\right)}du\right.\\
\left.+\tfrac{2\sqrt{\pi}p^{-k/2}}{\sqrt{1-x^2}}\tfrac{1}{2\pi i}\int_{(c)}\tilde{F}(u)\tfrac{\Gamma\left(\frac{1+u}{2}\right)}{\Gamma\left(\frac{2-u}{2}\right)}\left(\tfrac{\pi (4p^k)^{\alpha-1}}{(1-x^2)^{1-\alpha}}\right)^{-u}\tfrac{\zeta(2u)}{\zeta(u+1)}\tfrac{\left(1-p^{-u(k+1)}\right)}{\left(1-p^{-u}\right)}du\right]dx
\end{multline*}
Now note the following:
\begin{itemize}
\item $\tfrac{\Gamma\left(\frac{1+u}{2}\right)\zeta(2u)}{\Gamma\left(\frac{2-u}{2}\right)\zeta(u+1)}\tfrac{\left(1-p^{-u(k+1)}\right)}{\left(1-p^{-u}\right)}$ has a simple pole at $u=\tfrac{1}{2}$ with residue $\tfrac{1-p^{-(k+1)/2}}{2\zeta\left(\frac32\right)(1-p^{-1/2})}$ and is holomorphic on and to the right of the line $\Re(u)=0$.
\item By Lemma \ref{lemmaF}, $\tilde{F}(u)$ has a simple pole at $u=0$ with residue $1$ and is holomorphic otherwise. Note that in this case we also have $\lim_{u\rightarrow0}\tfrac{\Gamma\left(\frac{1+u}{2}\right)}{\Gamma\left(\frac{2-u}{2}\right)}\tfrac{\zeta(2u)}{\zeta(u+1)}\tfrac{\left(1-p^{-u(k+1)}\right)}{\left(1-p^{-u}\right)}=0$.
\item $\tfrac{\zeta(2u+2)\left(1-p^{-(u+1)(k+1)}\right)}{\zeta(u+2)\left(1-p^{-(u+1)}\right)}$ has a simple pole at $u=\tfrac{-1}{2}$ with residue $\tfrac{1-p^{-(k+1)/2}}{2\zeta(\frac32)(1-p^{-1/2})}$ and is holomorphic on and to the right of the line $\Re(u)=-1$.
\item The rest of the functions (of the variable $u$) in the first integral are holomorphic everywhere on and to the right of the line $\Re(u)=-1$, and in the second integral on and to the right of the line $\Re(u)=0$.
\end{itemize} 
Therefore by shifting the $u$-contour of the first integral to $\Re(u)=-1$ and the second to $\mathcal{C}_{\upsilon}$ we get,
\begin{align*}
\eqref{circ'}&=\tfrac{(4p^k)^{\frac{1-\alpha}{2}}\tilde{F}\left(\frac12\right)(1-p^{-(k+1)/2})}{2\zeta\left(\frac32\right)(1-p^{-1/2})}\int_{|x|<1}\tfrac{\theta_{\infty}^{+}(x)}{(1-x^2)^{\alpha/2}}dx+2p^{k/2}\tfrac{\left(1-p^{-(k+1)}\right)}{\left(1-p^{-1}\right)}\int_{|x|<1}\theta_{\infty}^{+}\left(x\right)dx\\
&+\tfrac{(4p^k)^{\frac{1-\alpha}{2}}\tilde{F}\left(\frac{-1}{2}\right)(1-p^{-(k+1)/2})}{2\zeta\left(\frac32\right)(1-p^{-1/2})}\int_{|x|<1}\tfrac{\theta_{\infty}^+(x)}{(1-x^2)^{\alpha/2}}dx\\
&+\tfrac{p^{k/2}}{2}\int_{|x|<1}\theta_{\infty}^{+}\left(x\right)\left[4\tfrac{1}{2\pi i}\int_{(-1)}\tilde{F}(u)\left(\tfrac{(4p^k)^{-\alpha}}{(1-x^2)^{\alpha}}\right)^{-u}\tfrac{\zeta(2u+2)}{\zeta(u+2)}\tfrac{\left(1-p^{-(u+1)(k+1)}\right)}{\left(1-p^{-(u+1)}\right)}du\right.\\
&\hspace{1.1in}\left.+\tfrac{2\sqrt{\pi}p^{-k/2}}{\sqrt{1-x^2}}\tfrac{1}{2\pi i}\int_{\mathcal{C}_{\upsilon}}\tilde{F}(u)\tfrac{\Gamma\left(\frac{1+u}{2}\right)}{\Gamma\left(\frac{2-u}{2}\right)}\left(\tfrac{\pi (4p^k)^{\alpha-1}}{(1-x^2)^{1-\alpha}}\right)^{-u}\tfrac{\zeta(2u)}{\zeta(u+1)}\tfrac{\left(1-p^{-u(k+1)}\right)}{\left(1-p^{-u}\right)}du\right]dx
\end{align*}
Finally recall that by Lemma \ref{lemmaF} $\tilde{F}$ is odd\footnote{We note that the oddness of $\tilde{F}$ is completely peripheral for the argument. The whole argument is valid for an arbitrary choice of $F$ and $\tilde{F}$. If $\tilde{F}$ is not odd, then we would get $-\tilde{F}(-u)$ in the dual part of the approximate functional equation.}, and therefore the first and the third terms above cancel and we get
\begin{align*}
\eqref{circ'}&=2p^{k/2}\tfrac{\left(1-p^{-(k+1)}\right)}{\left(1-p^{-1}\right)}\int_{|x|<1}\theta_{\infty}^{+}\left(x\right)dx\\
&+\tfrac{p^{k/2}}{2}\int_{|x|<1}\theta_{\infty}^{+}\left(x\right)\left[4\tfrac{1}{2\pi i}\int_{(-1)}\tilde{F}(u)\left(\tfrac{(4p^k)^{-\alpha}}{(1-x^2)^{\alpha}}\right)^{-u}\tfrac{\zeta(2u+2)}{\zeta(u+2)}\tfrac{\left(1-p^{-(u+1)(k+1)}\right)}{\left(1-p^{-(u+1)}\right)}du\right.\\
&\hspace{1.1in}\left.+\tfrac{2\sqrt{\pi}p^{-k/2}}{\sqrt{1-x^2}}\tfrac{1}{2\pi i}\int_{\mathcal{C}_{\upsilon}}\tilde{F}(u)\tfrac{\Gamma\left(\frac{1+u}{2}\right)}{\Gamma\left(\frac{2-u}{2}\right)}\left(\tfrac{\pi (4p^k)^{\alpha-1}}{(1-x^2)^{1-\alpha}}\right)^{-u}\tfrac{\zeta(2u)}{\zeta(u+1)}\tfrac{\left(1-p^{-u(k+1)}\right)}{\left(1-p^{-u}\right)}du\right]dx
\end{align*}

\item $x^2\pm1 >0$.

In this case we have,
\begin{multline*}
\tfrac{p^{k/2}}{2}\sum_{\mp}\sum_{f=1}^{\infty}\tfrac{1}{f^3}\sum_{l=1}^{\infty}\tfrac{Kl_{l,f}(0, p^k)}{l^2}\left\{\int_{x^2\pm1>0}\theta_{\infty}^{\mp}\left(x\right)\left[\tfrac{1}{2\pi i}\int_{(1)}\tilde{F}(u)\left(\tfrac{lf^2(4p^k)^{-\alpha}}{(x^2\pm1)^{\alpha}}\right)^{-u}du\right.\right.\\
\left.\left.+\tfrac{\sqrt{\pi}lf^2p^{-k/2}}{2\sqrt{x^2\pm1}}\tfrac{1}{2\pi i}\int_{(1)}\tilde{F}(u)\tfrac{\Gamma\left(\frac{u}{2}\right)}{\Gamma\left(\frac{1-u}{2}\right)}\left(\tfrac{\pi lf^2(4p^k)^{\alpha-1}}{(x^2\pm1)^{1-\alpha}}\right)^{-u}du\right]dx\right\} \tag{$\circ_2$}\label{circ''}
\end{multline*}
We proceed as above. Shifting the contour right to $\Re(u)=c>1$, then interchanging the $l$ and $f$-sums with the integrals and using Corollary \ref{pstcor9} results in

\begin{multline*}
\tfrac{p^{k/2}}{2}\sum_{\mp}\int_{x^2\pm1>0}\theta_{\infty}^{\mp}\left(x\right)\left[4\tfrac{1}{2\pi i}\int_{(c)}\tilde{F}(u)\left(\tfrac{(4p^k)^{-\alpha}}{(x^2\pm1)^{\alpha}}\right)^{-u}\tfrac{\zeta(2u+2)}{\zeta(u+2)}\tfrac{\left(1-p^{-(u+1)(k+1)}\right)}{\left(1-p^{-(u+1)}\right)}du\right.\\
\left.+\tfrac{2\sqrt{\pi}p^{-k/2}}{\sqrt{x^2\pm1}}\tfrac{1}{2\pi i}\int_{(c)}\tilde{F}(u)\tfrac{\Gamma\left(\frac{u}{2}\right)}{\Gamma\left(\frac{1-u}{2}\right)}\left(\tfrac{\pi (4p^k)^{\alpha-1}}{(x^2\pm1)^{1-\alpha}}\right)^{-u}\tfrac{\zeta(2u)}{\zeta(u+1)}\tfrac{\left(1-p^{-u(k+1)}\right)}{\left(1-p^{-u}\right)}du\right]dx
\end{multline*}
For what follows we will need to

Now note that:
\begin{itemize}
\item $\tfrac{\Gamma\left(\frac{u}{2}\right)\zeta(2u)}{\Gamma\left(\frac{1-u}{2}\right)\zeta(u+1)}\tfrac{\left(1-p^{-u(k+1)}\right)}{\left(1-p^{-u}\right)}$ has a simple pole at $u=\tfrac{1}{2}$ with residue $\tfrac{1-p^{-(k+1)/2}}{2\zeta\left(\frac32\right)(1-p^{-1/2})}$ and is holomorphic on and to the right of the line $\Re(u)=0$.
\item By Lemma \ref{lemmaF}, $\tilde{F}(u)$ has a simple pole at $u=0$ with residue $1$ and is holomorphic otherwise. On the other hand $\Gamma\left(\tfrac{u}{2}\right)$ has a simple pole with residue $2$ at $u=0$. Finally we see that $\tfrac{\zeta(2u)}{\zeta(u+1)}\tfrac{\left(1-p^{-u(k+1)}\right)}{\left(1-p^{-u}\right)}=u\zeta(2u)(k+1)+O(u^2)$ around $u=0$. Therefore $\tilde{F}(u)\tfrac{\Gamma\left(\frac{u}{2}\right)}{\Gamma\left(\frac{1-u}{2}\right)}\tfrac{\zeta(2u)}{\zeta(u+1)}\tfrac{\left(1-p^{-u(k+1)}\right)}{\left(1-p^{-u}\right)}$ has a simple pole at $u=0$ with residue $-\tfrac{k+1}{\sqrt{\pi}}$.
\item $\tfrac{\zeta(2u+2)\left(1-p^{-(u+1)(k+1)}\right)}{\zeta(u+2)\left(1-p^{-(u+1)}\right)}$ has a simple pole at $u=\tfrac{-1}{2}$ with residue $\tfrac{1-p^{-(k+1)/2}}{2\zeta(\frac32)(1-p^{-1/2})}$ and is holomorphic on and to the right of the line $\Re(u)=-1$.
\item The rest of the functions (of the variable $u$) in the first integral are holomorphic on and to the right of the line $\Re(u)=-1$, and in the second integral are holomorphic on and to the right of the line $\Re(u)=0$.
\end{itemize} 
Therefore shifting the first contour to $\Re(u)=-1$ and the second to $\mathcal{C}_{\upsilon}$ we get,
\begin{align*}
\eqref{circ''}&=\tfrac{(4p^k)^{\frac{1-\alpha}{2}}\tilde{F}\left(\frac12\right)(1-p^{-(k+1)/2})}{2\zeta\left(\frac32\right)(1-p^{-1/2})}\sum_{\mp}\int_{x^2\pm1>0}\tfrac{\theta_{\infty}^{\mp}(x)}{(x^2\pm1)^{\alpha/2}}dx+2p^{k/2}\tfrac{\left(1-p^{-(k+1)}\right)}{\left(1-p^{-1}\right)}\sum_{\mp}\int_{x^2\pm1>0}\theta_{\infty}^{\mp}\left(x\right)dx\\
&-(k+1)\sum_{\mp}\int_{x^2\pm1>0}\tfrac{\theta_{\infty}^{\mp}(x)}{\sqrt{x^2\pm1}}dx+\tfrac{(4p^k)^{\frac{1-\alpha}{2}}\tilde{F}\left(\frac{-1}{2}\right)(1-p^{-(k+1)/2})}{2\zeta\left(\frac32\right)(1-p^{-1/2})}\sum_{\mp}\int_{x^2\pm1>0}\tfrac{\theta_{\infty}^{\mp}(x)}{(x^2\pm1)^{\alpha/2}}dx\\
&+\tfrac{p^{k/2}}{2}\sum_{\mp}\int_{x^2\pm1>0}\theta_{\infty}^{\mp}\left(x\right)\left[4\tfrac{1}{2\pi i}\int_{(-1)}\tilde{F}(u)\left(\tfrac{(4p^k)^{-\alpha}}{(x^2\pm1)^{\alpha}}\right)^{-u}\tfrac{\zeta(2u+2)}{\zeta(u+2)}\tfrac{\left(1-p^{-(u+1)(k+1)}\right)}{\left(1-p^{-(u+1)}\right)}du\right.\\
&\left.\hspace{1.1in}+\tfrac{2\sqrt{\pi}p^{-k/2}}{\sqrt{x^2\pm1}}\tfrac{1}{2\pi i}\int_{\mathcal{C}_{\upsilon}}\tilde{F}(u)\tfrac{\Gamma\left(\frac{u}{2}\right)}{\Gamma\left(\frac{1-u}{2}\right)}\left(\tfrac{\pi (4p^k)^{\alpha-1}}{(x^2\pm1)^{1-\alpha}}\right)^{-u}\tfrac{\zeta(2u)}{\zeta(u+1)}\tfrac{\left(1-p^{-u(k+1)}\right)}{\left(1-p^{-u}\right)}du\right]dx
\end{align*}

The first and fourth terms in the above sum cancel because $\tilde{F}$ is odd (Once again this is not essential to the argument. See the footnote above.), and we get

\begin{align*}
\eqref{circ''}&=2p^{k/2}\tfrac{\left(1-p^{-(k+1)}\right)}{\left(1-p^{-1}\right)}\sum_{\mp}\int_{x^2\pm1>0}\theta_{\infty}^{\mp}\left(x\right)dx-(k+1)\sum_{\mp}\int_{x^2\pm1>0}\tfrac{\theta_{\infty}^{\mp}(x)}{\sqrt{x^2\pm1}}dx\\
&+\tfrac{p^{k/2}}{2}\sum_{\mp}\int_{x^2\pm1>0}\theta_{\infty}^{\mp}\left(x\right)\left[4\tfrac{1}{2\pi i}\int_{(-1)}\tilde{F}(u)\left(\tfrac{(4p^k)^{-\alpha}}{(x^2\pm1)^{\alpha}}\right)^{-u}\tfrac{\zeta(2u+2)}{\zeta(u+2)}\tfrac{\left(1-p^{-(u+1)(k+1)}\right)}{\left(1-p^{-(u+1)}\right)}du\right.\\
&\left.\hspace{1.1in}+\tfrac{2\sqrt{\pi}p^{-k/2}}{\sqrt{x^2\pm1}}\tfrac{1}{2\pi i}\int_{\mathcal{C}_{\upsilon}}\tilde{F}(u)\tfrac{\Gamma\left(\frac{u}{2}\right)}{\Gamma\left(\frac{1-u}{2}\right)}\left(\tfrac{\pi (4p^k)^{\alpha-1}}{(x^2\pm1)^{1-\alpha}}\right)^{-u}\tfrac{\zeta(2u)}{\zeta(u+1)}\tfrac{\left(1-p^{-u(k+1)}\right)}{\left(1-p^{-u}\right)}du\right]dx
\end{align*}

\end{itemize}

Summing \eqref{circ'} and \eqref{circ''} finishes the proof.

\end{proof}

Finally we have the following auxiliary lemma that identifies the contribution of the special representations in the sum in Theorem \ref{6thm1}.

\begin{lemma}\label{6cor1}Let $\tr(\textbf{1}(f^{p,k}))$ is the contribution of the trivial representation, and $\tr(\xi_0(f^{p,k}))$ is the contribution to the trace formula by the residues of the Eisenstein series as explained on pg.25 of \cite{L1}. Then,

\begin{align*}
\tr(\textbf{1}(f^{p,k}))&=2p^{k/2}\tfrac{\left(1-p^{-(k+1)}\right)}{\left(1-p^{-1}\right)}\sum_{\mp}\int\theta_{\infty}^{\mp}\left(x\right)dx\\
\tr(\xi_0(f^{p,k}))&=\tfrac{k+1}{2}\sum_{\mp}\int_{x^2\pm1>0}\tfrac{\theta_{\infty}^{\mp}(x)}{\sqrt{|x^2\pm1|}}dx\\
\end{align*}

\end{lemma}

\begin{proof}

We start with the trivial representation. Recall \eqref{2starstar} that $\theta_{\infty}^{\mp}(x)=2|x^2\pm|^{1/2}g_1^{\mp}(x)+g_2^{\mp}(x)$, where $\gamma_{(\mp1,x)}$ is as in \eqref{2star}. Then,
\begin{align*}
2p^{k/2}\tfrac{\left(1-p^{-(k+1)}\right)}{\left(1-p^{-1}\right)}\sum_{\mp}\int\theta_{\infty}^{\mp}\left(x\right)dx&=2p^{k/2}\tfrac{\left(1-p^{-(k+1)}\right)}{\left(1-p^{-1}\right)}\sum_{\mp}\int (2|x^2\pm1|^{1/2}g_1^{\mp}(x)+g_2^{\mp}(x))dx\\
&=4p^{k/2}\tfrac{\left(1-p^{-(k+1)}\right)}{\left(1-p^{-1}\right)}\sum_{\mp}\int \left(g_1^{\mp}(x)+\tfrac{g_2^{\mp}(x)}{2|x^2\pm1|^{1/2}}\right)|x^2\pm1|^{1/2}dx\tag{$*$}\label{triv}
\end{align*}

Now a quick comparison of \eqref{triv} with equation\footnote{There is a misprint in equation (65) of \cite{L1}, the exponent $m$ should be $m+1$.}  (65) of \cite{L1} (using equation (26) of the same reference) shows that $\eqref{triv}=\tr(\textbf{1}(f^{p,k}))$.

For the second equality we only need to note that the integer we denote by $k$ is denoted by $m$ in \cite{L1}, and equation (31) of \cite{L1} is equal to $\tfrac{k+1}{2}\sum_{\mp}\int_{x^2\pm1>0}\tfrac{\theta_{\infty}^{\mp}(x)}{\sqrt{|x^2\pm1|}}dx$.

\end{proof}

Finally Theorem \ref{6thm1} combined with Lemma \ref{6cor1} finishes the proof of Theorem \ref{mainthm}.

\end{section}

\vspace{0.4in}

Salim Ali Altu\u{g}\hspace{0.4in} \textsf{altug@math.columbia.edu}\\
Mathematics Department, Columbia University, 2990 Broadway, New York, NY, USA 10027\\

\end{document}